\documentclass[11pt,a4paper,oneside]{article}
\usepackage[margin=1in]{geometry}
\usepackage[usenames,dvipsnames]{xcolor}
\usepackage[pdftitle={Asymptotic normality of consecutive patterns in permutations encoded by generating trees with one-dimensional labels},
pdfauthor={Jacopo Borga},
colorlinks=true,linkcolor=JungleGreen,urlcolor=OliveGreen,citecolor=PineGreen,bookmarks=true,bookmarksopen=true,bookmarksopenlevel=2,unicode=true,linktocpage]{hyperref}

\usepackage{amssymb,amsfonts}
\usepackage{enumerate}
\usepackage{amsmath}
\usepackage{amsmath,bm}
\usepackage{mathtools}
\usepackage{amsthm}
\usepackage{dsfont}
\usepackage{tikz}
\usepackage[capitalize]{cleveref}
\usetikzlibrary{babel}
\usepackage{authblk}  

\usepackage{graphicx}  
\usepackage[utf8]{inputenc} 
\usepackage{caption}
\newtheorem{thm}{Theorem}[section]
\newtheorem{cor}[thm]{Corollary}
\newtheorem{prop}[thm]{Proposition}
\newtheorem{lem}[thm]{Lemma}

\theoremstyle{definition}
\newtheorem{defn}[thm]{Definition}

\newtheorem{exmp}[thm]{Example}

\newtheorem{obs}[thm]{Observation}
\newtheorem{ass}{Assumption}

\theoremstyle{remark}
\newtheorem{rem}[thm]{Remark}

\makeatletter
\newcommand\mathcircled[1]{%
	\mathpalette\@mathcircled{#1}%
}
\newcommand\@mathcircled[2]{%
	\tikz[baseline=(math.base)] \node[draw,circle,inner sep=1pt] (math) {$\m@th#1#2$};%
}
\makeatother

\usepackage{todonotes}

\def\Sr{\mathcal{S}_{\bullet}}
\def\Sri{\tilde{\mathcal{S}_{\bullet}}}

\def\Z{\mathbb{Z}}
\def\E{\mathbb{E}}
\def\P{\mathbb{P}}
\def\coc{c\text{-}occ}
\def\Asi{A_{\sigma,i}}
\def\leqsi{\preccurlyeq_{\sigma,i}}

\DeclareMathOperator{\mult}{mult}

\DeclareMathOperator{\GG}{G}
\DeclareMathOperator{\CL}{CL}

\DeclareMathOperator{\pat}{pat}
\DeclareMathOperator{\Pat}{Pat}

\DeclareMathOperator{\Av}{Av}
\DeclareMathOperator{\Var}{Var}
\DeclareMathOperator{\Cov}{Cov}

\newcounter{indice}

\title{Asymptotic normality of consecutive patterns in permutations encoded by generating trees with one-dimensional labels}

\date{  }
\author{Jacopo Borga\footnote{\href{mailto:jacopo.borga@math.uzh.ch}{jacopo.borga@math.uzh.ch}}}
\affil{Institut für Mathematik, Universität Zürich}

\makeatletter
\newcommand{\subjclass}[2][1991]{%
	\let\@oldtitle\@title%
	\gdef\@title{\@oldtitle\footnotetext{#1 \emph{Mathematics subject classification.} #2}}%
}
\newcommand{\keywords}[1]{%
	\let\@@oldtitle\@title%
	\gdef\@title{\@@oldtitle\footnotetext{\emph{Key words and phrases.} #1.}}%
}
\makeatother



\begin{document}

\maketitle

\begin{abstract}
We consider uniform random permutations drawn from a family enumerated through generating trees. We develop a new general technique to establish a central limit theorem for the number of consecutive occurrences of a fixed pattern in such permutations. 

We propose a technique to sample uniform permutations in such families as conditioned random colored walks. Building on that, we derive the behavior of the consecutive patterns in random permutations studying properties of the consecutive increments in the corresponding random walks. The method applies to families of permutations with a one-dimensional-labeled generating tree (together with some technical assumptions) and implies local convergence for random permutations in such families. We exhibit ten different families of permutations, most of them being permutation classes, that satisfy our assumptions.

To the best of our knowledge, this is the first work where generating trees -- which were introduced to enumerate combinatorial objects -- have been used to establish probabilistic results.
\end{abstract}


\section{Introduction}

\subsection{Generating trees: a new probabilistic approach}
The theory of generating trees has been widely used to enumerate families of combinatorial objects. We give here an informal definition of generating trees and we refer the reader to \cref{sect:gen_tree_perm} for a more precise introduction to them in the specific setting of permutations.

A \emph{generating tree} for a combinatorial class $\mathcal{C}$ is an infinite rooted tree whose vertices are the elements of $\mathcal{C}$ (each appearing exactly once in the tree) and such that the objects of size $n$ are at level $n$ (the level of a vertex being the distance from the root plus one).
The children of some object $\square\in\mathcal{C}$ correspond to the objects obtained by adding a new \emph{``atom"} to $\square$, i.e.\ a piece of object which make the size increase by $1$.

Although generating trees first appeared in  the literature in the context of permutations with forbidden patterns \cite{MR491652,MR1360119,MR1417303, MR1630680}, they were rigorously defined later, together with the ECO method (Enumerating Combinatorial Objects), in \cite{MR1717162}. The latter is a technique that is based on a recursive construction of the objects of a combinatorial class and that provides a useful tool to establish enumerative results. 

Generating trees have been used in the last 20 years to establish several enumerative results for various combinatorial classes of partitions, permutations, polyominoes and many other objects (see for instance \cite{bousquet2003four, marinov2003generating, MR2209164, MR2376115, MR2880652, MR3537914, MR3882946, MR3861052, duchi2019code, MR3961884,beaton2015slicings}). We refer to \cite{banderier2002generating} and to the Ph.D.\ thesis of Guerrini~\cite[Chapter 1]{guerrini2017enumeration} for two interesting presentations of generating trees and associated enumeration techniques through generating functions.

The goal of this article is to introduce a new facet of generating trees encoding families of permutations, in order to establish probabilistic results instead of enumerative ones. We specifically focus on central limit theorems for consecutive occurrences of patterns in permutations and their consequences on local limits for permutations. 

\subsection{Central limit theorems and local limits for random permutations}

Concentration and/or fluctuation results, like laws of large numbers or central limit theorems (CLTs), are some of the most classical results in probability theory. In the specific case of random permutations, the behavior of various statistics has been studied in many works. We mention a few of them.
\begin{itemize}
	\item One of the most famous and old CLTs for statistics of permutations can be found in the work of Goncharov~\cite{goncharov1944some, MR0131369}, where the author showed that the number of cycles of a uniform random permutation is asymptotically Gaussian after suitable normalization.
	
	\item CLTs for pattern occurrences in uniform
	random permutations were established by Fulman~\cite{MR2118603} (for inversions and descents), Goldstein~\cite[see in particular Example 3.2]{MR2157512}
	(for consecutive patterns), Bóna~\cite{MR2732825} (for monotone patterns, both in the consecutive and classical settings), Janson, Nakamura and Zeilberger~\cite{MR3338847} (for general classical patterns) and Hofer~\cite{MR3783365} (for vincular patterns).
	
	\item Janson~\cite{janson2017patterns, MR3983781, MR4058419} recently considered a random permutation drawn from the set of
	permutations of size $n$ that avoid some given set of patterns of size
	$3$. He showed that the number of occurrences of another pattern has
	a limit in distribution (not always Gaussian), after suitable scaling.
	
	\item F{\'e}ray~\cite{MR3091722} studied the asymptotic behavior of some statistics (including the number of occurrences of any vincular pattern) in Ewens distributed random permutations. Moreover, in \cite{feray2018central} he used the recently developed method of weighted dependency graphs \cite{MR3858921} to prove CLTs for the number of occurrences of any fixed pattern in multiset permutations and in set partitions.
	
	\item Basu and Bhatnagar \cite{MR3729641} studied the lengths of the longest monotone subsequences in permutations drawn from the Mallows measure. This is a non-uniform model where the probability of every permutation is proportional to $q^{\text{number of inversions}}$. They showed that when $0<q<1$, then the limiting distribution of the longest increasing subsequence (LIS) is Gaussian. This is in contrast with the uniform case (when $q=1$) where the limiting distribution of the LIS, when scaled appropriately, is the GUE Tracy–Widom distribution, as shown by Baik, Deift and  Johansson~\cite{MR1682248}. Furthermore, Crane, De Salvo
	and Elizalde \cite{MR3854041} obtain a CLT governing the number of consecutive occurrences of a given pattern in Mallows distributed random permutations.
\end{itemize}

Under some specific assumptions (made precise later), we obtain a CLT for the number of consecutive occurrences of patterns in permutations encoded by generating trees (see \cref{thm:main_thm_CLT} below). The choice of studying consecutive patterns is motivated by a recent work of the author~\cite{borga2018local}, where a notion of local topology for permutations was introduced. In this work, the author has shown that local convergence for a sequence of permutations is equivalent to the convergence of the proportions of consecutive patterns (we refer to \cref{sect:appe_local_topo} for an introduction to local topology for permutations). 
We mention that local convergence for random permutations has been also recently investigated in \cite{bevan2019permutations,borga2019square,borga2018localsubclose,Baxterscallim}.

In \cref{corl:main_thm} below, we use the aforementioned characterization of the local convergence to show that all the families of permutations that verify the CLT in \cref{thm:main_thm_CLT} locally converge both in the \emph{annealed} and in the \emph{quenched} sense\footnote{These two different versions of local convergence come from the fact that there are two sources of randomness, one for the choice of the random permutation, and one for the random root. Intuitively, in the annealed version, the random permutation and the random root are taken simultaneously, while in the quenched version, the random permutation should be thought as frozen when we take the random root. We refer to Definitions \ref{defn:weakweakconv} and \ref{strongconv} for further details.}. We highlight that in order to establish quenched local convergence (which implies the annealed one) for a sequence of random permutations, it is actually enough to establish a law of large numbers for the corresponding proportion of consecutive patterns (see \cref{thm:local_conv_perm_charact}). 

\medskip

As a further motivation for our work, we point out that the study of consecutive patterns, started by Elizalde and Noy~\cite{MR1979785,MR3017964}, has received significant attention in combinatorics and probability  in the last 15 years, with connections with other areas of mathematics,
such as dynamical systems, and other fields, such as computer science, biology, and
physics; see \cite{MR3526425} for a survey.

\subsection{An informal presentation of the main result}

We now briefly explain the CLT result presented in \cref{thm:main_thm_CLT} and the strategy to prove it. For a more precise formulation and explanation, we need some more preliminary results, and we refer the reader to \cref{sect:gen_tree_perm,sect:main_results}.
\begin{itemize}
	\item We first establish a size-preserving bijection between permutations encoded by a generating tree, whose set is denoted by $\mathcal C$, and some colored walks, whose set is denoted by $\mathcal{W}$ (see \cref{sect:bij_color_walk_perm}).
	\item We make the following three assumptions:
	\begin{itemize}
		\item We restrict our analysis to one-dimensional-labeled generating trees in such a way that the corresponding walks in $\mathcal W$ are one-dimensional (see \cref{ass1}).
		\item We assume that we can sample a uniform walk in $\mathcal W$ as a random walk conditioned to stay positive (see \cref{ass2}).
		\item We assume that the consecutive patterns of a permutation in $\mathcal C$ can be recovered from the consecutive increments of the corresponding walk in $\mathcal W$ (see \cref{ass3}).
	\end{itemize}
	\cref{sect:sampling_uniformly} is dedicated to developing a strategy for checking \cref{ass2}: given a family of permutations $\mathcal{C}$ encoded by a one-dimensional-labeled generating tree, we present a technique for sampling a uniform permutation in $\mathcal{C}$ as a random colored walk conditioned to stay positive.
	\item We then study the behavior of the conditioned random walk in order to establish a CLT for the number of consecutive occurrences of any fixed family of increments (see \cref{sect:CTL_increments}, specifically \cref{lem:lemma3}).
	\item Finally, using our assumptions, we transfer the above result to permutations (see \cref{sect:main_thm_proof}) proving a CLT for the number of  consecutive occurrences of any pattern in uniform permutations in $\mathcal C$. 
\end{itemize}

In \cref{sect:examples_gen_tree_Ok} we collect several families of permutations that satisfy Assumptions \ref{ass1}, \ref{ass2} and \ref{ass3}, proving \cref{thm:examples_ok} below.

\medskip

Let $\mathcal{S}$ be the set of all permutations of finite size. Given a set of (generalized) patterns $B$ we denote by $\Av(B)$ the set of $B$-avoiding permutations. We also denote by $\coc(\pi,\sigma)$ the number of consecutive occurrences of a pattern $\pi$ in a permutation $\sigma.$ We refer the reader to \cref{sect:notation on perm} for more definitions on permutation patterns.

\begin{thm}\label{thm:examples_ok}
	Let $\mathcal C$ be one of the following families of permutations\footnote{We refer the reader to \cref{sect:strange_patterns} for the definition of the last two families of permutations in the list, which avoid generalized patterns.}
	\begin{align*}
	\Av(123),\quad
	\Av(132),\quad
	\Av(1423,&4123),\quad
	\Av(1234,2134),\quad
	\Av(1324,3124),\\
	\Av(2314,3214),\quad
	\Av(2413,4213),\quad
	\Av&(3412,4312),\quad
	\Av(213,\bar{2}\underbracket[.5pt][1pt]{31}),\quad
	\Av(213,\bar{2}^{o}\underbracket[.5pt][1pt]{31}).
	\end{align*}
	For all $n\in\Z_{>0},$ let $\bm{\sigma}_n$ be a uniform random permutation in $\mathcal{C}$ of size $n$. 
	Then, for all patterns $\pi\in\mathcal S$, we have the following central limit theorem,
	\begin{equation*}
	\frac{\coc(\pi,\bm{\sigma}_n)-n\mu_{\pi}}{\sqrt n}\stackrel{d}{\longrightarrow}\bm{\mathcal N}(0,\gamma_{\pi}^2),
	\end{equation*}
	where $\mu_{\pi}$ and $\gamma_{\pi}^2$ are explicitly described in \cref{thm:main_thm_CLT}.
\end{thm}

\begin{rem}
	We point out that a law of large numbers for the families $\Av(123)$ and $\Av(132)$ has been established by the author in \cite[Theorem 1.8,Theorem 1.9]{borga2018local}. Nevertheless the central limit theorem in \cref{thm:examples_ok} is a new result also for these two families.
\end{rem}

\subsection{Permutations and patterns}
\label{sect:notation on perm}
We introduce some notation. For any $n\in\Z_{>0},$ we denote the set of permutations of $[n]=\{1,2,\dots,n\}$ by $\mathcal{S}^n.$ We write permutations of $\mathcal{S}^n$ in one-line notation as $\sigma=\sigma(1)\sigma(2)\dots\sigma(n).$ Sometimes, we view a permutation $\sigma$ of size $n$ as a set of points in $[n]\times[n]$ where the $i$-th column has exactly one point at height $\sigma(i)$. We call this representation the \emph{diagram} of a permutation. For a permutation $\sigma\in\mathcal{S}^n$ the \emph{size} $n$ of $\sigma$ is denoted by $|\sigma|.$ Recall that $\mathcal{S}\coloneqq\bigcup_{n\in\Z_{>0}}\mathcal{S}^n$ denotes the set of permutations of finite size. We write sequences of permutations in $\mathcal{S}$ as $(\sigma_n)_{n\in\Z_{>0}}.$ 

If $x_1,\dots ,x_n$ is a sequence of distinct numbers, let $\text{std}(x_1,\dots ,x_n)$ be the unique permutation $\pi$ in $\mathcal{S}^n$ that is in the same relative order as $x_1,\dots ,x_n,$ i.e.\ $\pi(i)<\pi(j)$ if and only if $x_i<x_j.$
Given a permutation $\sigma\in\mathcal{S}^n$ and a subset of indices $I\subset[n]$, let $\pat_I(\sigma)$ be the permutation induced by $(\sigma(i))_{i\in I},$ namely, $\pat_I(\sigma)\coloneqq\text{std}\big((\sigma(i))_{i\in I}\big),$ where the $\sigma(i)$ are listed in increasing value of the index.
For example, if $\sigma=87532461$ and $I=\{2,4,7\}$ then $\pat_{\{2,4,7\}}(87532461)=\text{std}(736)=312$.

Given two permutations, $\sigma\in\mathcal{S}^n$ for some $n\in\Z_{>0}$ and $\pi\in\mathcal{S}^k$ for some $k\leq n,$ and a set of indices $I=\{i_1 < \ldots < i_k\}$, we say that $\sigma(i_1) \ldots \sigma(i_k)$ is an \emph{occurrence} of $\pi$ in $\sigma$ if $\pat_I(\sigma)=\pi$ (we will also say that $\pi$ is a \emph{pattern} of $\sigma$). If the indices $i_1, \ldots ,i_k$ form an interval, then we say that $\sigma(i_1) \ldots \sigma(i_k)$ is a \emph{consecutive occurrence} of $\pi$ in $\sigma$ (we will also say that $\pi$ is a \emph{consecutive pattern} of $\sigma$).
We denote intervals of integers as $[n,m]$ for some $n,m\in\Z_{>0}$ with $n\leq m$.

\begin{exmp} 
	The permutation $\sigma=1532467$ contains an occurrence (but no consecutive occurrence) of $1423$ and a consecutive occurrence of $321$. Indeed $\pat_{\{1,2,3,5\}}(\sigma)=1423$ but no interval of indices of $\sigma$ induces the permutation $1423.$ Moreover, $\pat_{[2,4]}(\sigma)=\pat_{\{2,3,4\}}(\sigma)=321.$
\end{exmp}
We say that $\sigma$ \emph{avoids} $\pi$ if $\sigma$ does not contain any occurrence of $\pi$. We point out that the definition of $\pi$-avoiding permutations refers to occurrences and not to consecutive occurrences. Given a set of patterns $B\subset\mathcal{S},$ we say that $\sigma$ \emph{avoids} $B$ if $\sigma$ avoids $\pi,$ for all $\pi\in B$. We denote by $\Av^n(B)$ the set of $B$-avoiding permutations of size $n$ and by $\Av(B)\coloneqq\bigcup_{n\in\Z_{>0}}\Av^n(B)$ the set of $B$-avoiding permutations of arbitrary finite size\footnote{For a set $B$ of classical patterns, the set $\Av(B)$ is called \emph{permutation class}. Since in the paper we will also consider permutations avoiding generalized patterns, we prefer not to use this terminology in order to avoid any confusion.}. 
In general, if $\mathcal{C}$ is a family of permutations, we denote by $\mathcal{C}^n$ the set of permutations of size $n$ in $\mathcal{C}$.

\medskip

We end this section with some remarks on the probabilistic notation that we adopt in the paper. We denote random quantities using \textbf{bold} characters. Given a random variable $\bm{X}$ we denote its law with $\mathcal{L}aw(\bm{X})$. We write $\xrightarrow{P}$ for convergence in probability and $\xrightarrow{d}$ for convergence in distribution. Finally $\stackrel{d}{=}$ denotes an equality in distribution.

\subsection{Generating trees for families of permutations}\label{sect:gen_tree_perm}

We introduce here the notion of generating tree and we collect some useful observations and results about these structures which are needed to precisely state our main results in \cref{sect:main_results}. 

\subsubsection{Definitions and examples}
Since we are interested in the study of permutations, we restrict the definition of generating tree to these specific objects. We need the following preliminary construction.

\begin{defn}
	Given a permutation $\sigma\in\mathcal{S}^n$ and an integer $m\in [n+1],$ we denote by $\sigma^{*m}$ the permutation obtained from $\sigma$ by \emph{appending a new final value} equal to $m$ and shifting by $+1$ all the other values larger than or equal to $m.$ Equivalently,
	$$\sigma^{*m}\coloneqq\text{std}(\sigma(1),\dots,\sigma(n),m-1/2).$$
\end{defn}

We say that a family of permutations $\mathcal{C}$ \emph{grows on the right} if for every permutation $\sigma\in\mathcal{S}$ such that $\sigma^{*m}\in\mathcal{C}$ for some $m\in [|\sigma|+1]$, we have $\sigma\in\mathcal{C}.$ From now until the end of the article we will only consider families of permutations that grow on the right, without explicitly saying it every time.

\begin{defn}
	\label{def:gen_tree}
	The \emph{generating tree} for a family of permutations $\mathcal{C}$ is the infinite rooted tree whose vertices are the permutations of $\mathcal{C}$ (each appearing exactly once in the tree) and such that permutations of size $n$ are at level $n$ (the level of a vertex being the distance from the root plus one).
	The children of some permutation $\sigma\in\mathcal{C}$ correspond to the permutations obtained by appending a new final value to $\sigma$.
\end{defn}

\begin{rem}
	Another classical way of defining generating trees for families of permutations is by appending a new maximal value on the top of the diagram of $\sigma$. We will never consider this case in our article since it is less convenient for our purposes of studying consecutive patterns.
\end{rem}

We now give an example of a generating tree for the family of permutations avoiding the patterns 1423 and 4123 inspired by a result of Kremer \cite{kremer2000permutations, kremer2003postscript}\footnote{Note that the various results in \cite{kremer2000permutations, kremer2003postscript} are stated for families of permutations that grows on the top and not on the right. However, these results  can be translated into permutations growing on the right by using the diagonal symmetry of the diagram of pattern-avoiding permutations. For instance, we can recover the result for $\{1423,4123\}$-avoiding permutations growing on the right from $\{1342,2341\}$-avoiding permutations growing on the top.}. This family will appear several times along this section (see Examples \ref{exmp:1423_4123_av_part0}, \ref{exmp:1423_4123_av} and \ref{exmp:1423_4123_av_part2}). We inform the reader that several results presented in these examples are useful also for Section \ref{sec:1423,4123}.

\begin{exmp}
	\label{exmp:1423_4123_av_part0}
	The first levels of the generating tree for $\{1423,4123\}$-avoiding permutations are drawn in \cref{gen_tree_1423_4123}. Note that every child of a permutation is obtained by appending a new final value. For instance, the permutation $321$ on the left-hand side of the third level is obtained from the permutation $21$ by appending a new final value equal to $1.$
\end{exmp}

\begin{defn}
	\label{def:gen_tree_part2}
	Let $S$ be a set and assume there exists an $S$-valued statistic on the permutations of $\mathcal{C},$ whose value determines the number of children in the generating tree and the values of the statistic for the children. Then we give labels to the objects of $\mathcal{C}$ which indicate the value of the statistic. The associated \emph{succession rule} is then given by the label $\lambda$ of the root and, for any label $k,$ the labels $e_1(k),\dots,e_{h(k)}(k)$ of the $h(k)$ children of an object labeled by $k.$ In full generality, the associated succession rule is represented by the formula
	\begin{equation}
	\begin{cases} 
	\text{Root label}: (\lambda) \\
	(k)\to (e_1{\scriptstyle(k)}),\dots,(e_{h(k)}{\scriptstyle(k)}).
	\end{cases}
	\label{eq:ass_succ_rul}
	\end{equation}
	The set of labels appearing in a generating tree is denoted by $\mathcal{L}$ and for all $k\in\mathcal{L},$ the multiset of labels of the children $\{e_1(k),\dots,e_{h(k)}(k)\}$ is denoted by $\CL(k),$ where $\CL$ stands for \emph{``children labels"}.
\end{defn}

\begin{rem} \label{rem:multi-dim}
	In our examples, we only consider generating trees with labels taking values in some subset of $\mathbb{Z}_{>0}.$ However, other types of generating trees have been considered and studied, for instance trees with $(\mathbb{Z}_{>0})^2$-valued labels (an example is the generating tree for the well-known Baxter permutations \cite{bousquet2003four}) or (integer tuples)-valued labels (see \cite{marinov2003generating}). 
\end{rem}

Before continuing Example \ref{exmp:1423_4123_av_part0}, we introduce some terminology.

\begin{defn}
	Given a family of permutations $\mathcal{C},$ a permutation $\sigma\in\mathcal{C}^n$ and an integer $m\in [n+1],$ we say that $m$ is an \emph{active site} of $\sigma$ if $\sigma^{*m}\in\mathcal{C}.$ We denote by $\text{AS}(\sigma)$ the set of active sites of $\sigma$ and by $|\text{AS}(\sigma)|$ its cardinality.
	If $\text{AS}(\sigma)=\{i_1,\dots,i_k\},$ with $i_1<\dots<i_k,$ we call $i_1$ the \emph{bottom active site} and $i_k$ the \emph{top active site}.
\end{defn}

\begin{rem}
	In our examples, we only consider generating trees whose succession rule is determined (in the sense of \cref{def:gen_tree_part2}) by the statistic $|\text{AS}(\cdot)|$. However,  in other examples, we need to consider another statistics. For instance, the succession rule for the generating tree for Baxter permutations is determined by the number of left-to-right and right-to-left maxima (see \cite{bousquet2003four}).
\end{rem}

\begin{exmp}\label{exmp:1423_4123_av}
	Given a permutation $\sigma\in\Av^n(1423,4123),$ we have the following properties for the set $\text{AS}(\sigma)$ (for a proof see \cite[Theorem 8]{kremer2000permutations}):
	\begin{itemize}
		\item 1,2 and $n+1$ are always active sites, i.e.\  $\{1,2,n+1 \}\subseteq\text{AS}(\sigma)$.
		\item If $\text{AS}(\sigma)=\{i_1=1,i_2=2,i_3,\dots,i_{k-1},i_k=n+1\}$ then
		\begin{itemize}
			\item $\text{AS}(\sigma^{*1})=\{1,2,3,i_3+1,i_4+1,\dots,i_{k-1}+1,n+2\}$;
			\item $\text{AS}(\sigma^{*(n+1)})=\{1,2,i_3,\dots,i_{k-1},n+1,n+2\}$;
			\item $\text{AS}(\sigma^{*i_j})=\{1,2,i_3,\dots,i_{j},n+2\},$ for all $2\leq j< k.$
		\end{itemize}
	\end{itemize}
	
	Note that, as a consequence, $|\text{AS}(\sigma^{*1})|=|\text{AS}(\sigma^{*(n+1)})|=|\text{AS}(\sigma)|+1,$
	while,
	$|\text{AS}(\sigma^{*i_j})|=j+1,$ for $j\in[2,k-1]$.
	
	In  particular, the statistic $|\text{AS}(\cdot)|$ fulfills the conditions of \cref{def:gen_tree_part2} for  the family of $\{1423,4123\}$-avoiding permutations, and the associated succession rule is
	\begin{equation}
	\begin{cases} 
	\text{Root label}: (2) \\
	(k)\to (k+1),(3),(4)\dots,(k),(k+1). 
	\end{cases}
	\label{eq:succ_rule_1423_4123}
	\end{equation}
	
	For convenience, we choose to write the succession rule and to draw the children of a permutation (see \cref{gen_tree_1423_4123} below) in increasing order of the position of the appended final value. More precisely, if a permutation $\sigma\in\Av^n(1423,4123)$ has active sites $\text{AS}(\sigma)=\{i_1,\dots,i_k\},$ with $i_1<\dots<i_k$, its children $\sigma^{*i_1},\dots, \sigma^{*i_k}$ (with labels $(k+1),(3),(4)\dots,(k),(k+1)$) are drawn in this order from left to right.
	
	\begin{figure}[htbp]
		\begin{center}
			\includegraphics[scale=0.6]{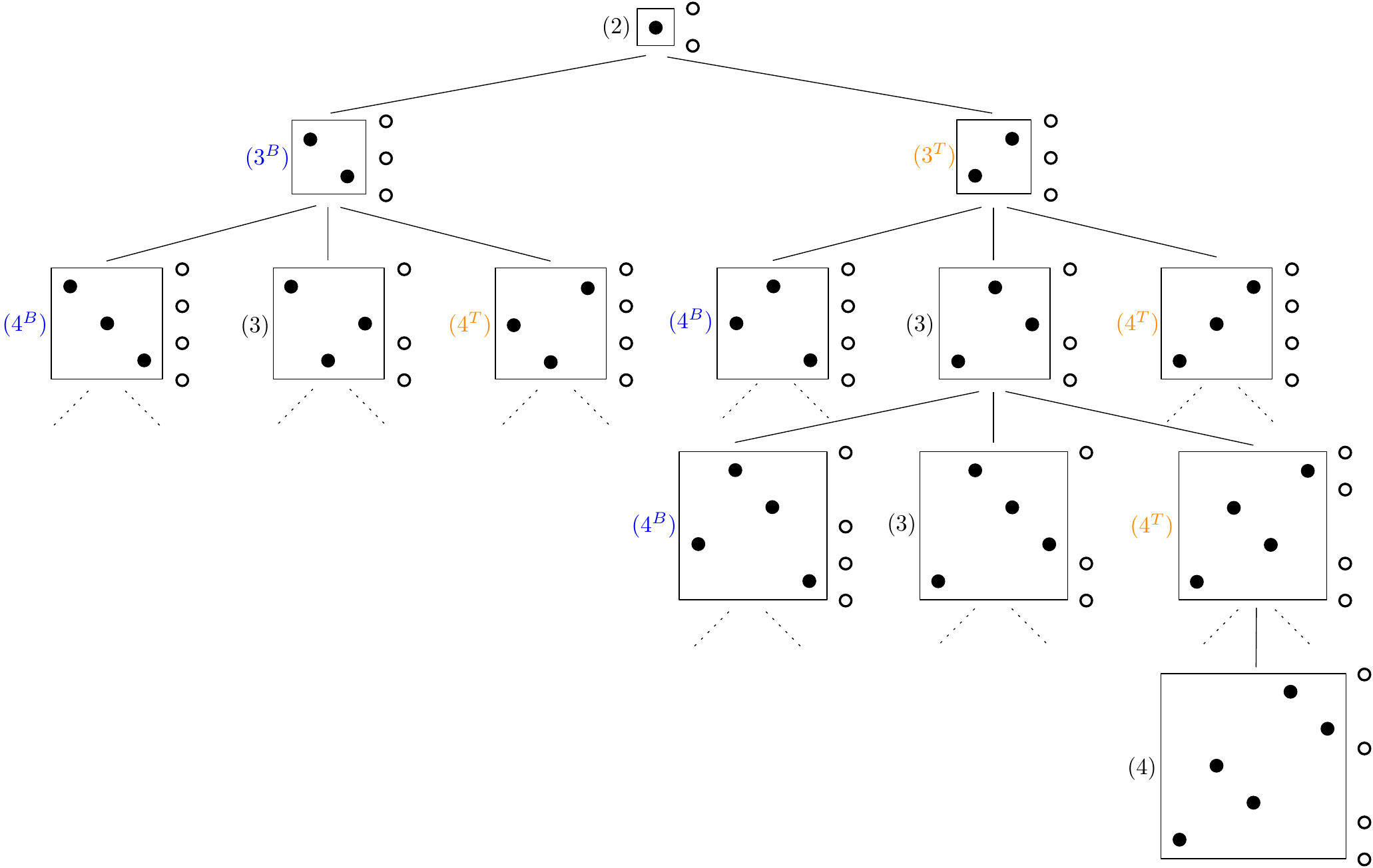}\\
			\caption{The generating tree for $\{1423,4123\}$-avoiding permutations. We completely draw only the first three levels of the tree; for the fourth level, we only draw the children of the permutation 132, and for the fifth level only one of the children of the permutation 1324. For each permutation, we highlight the position of the active sites with some circles on the right-hand side of the diagram. Moreover, on the left-hand side of each diagram we report the corresponding label given by the statistic $|\text{AS}(\cdot)|$. The superscripts $T$ and $B$ and the colors that appear in some labels will be clarified in Example \ref{exmp:1423_4123_av_part2}. \label{gen_tree_1423_4123}}
		\end{center}
	\end{figure}
\end{exmp}

\subsubsection{A bijection between permutations encoded by a generating tree and colored walks}
\label{sect:bij_color_walk_perm}

We start by pointing out that we are not assuming that the children labels $e_1(k),\dots,e_{h(k)}(k)$ appearing in the succession rule in \cref{eq:ass_succ_rul} are distinct (notice for instance that the label $(k+1)$ is produced twice from the label $(k)$ in \cref{eq:succ_rule_1423_4123}). For each pair of labels $(k,k')\in\mathcal{L}^2$, we denote by $\mult_{k}(k')$ the number of indices $i$ such that $ e_{i}(k)=k'.$

There is a natural bijection between permutations in a family encoded by a generating tree and the set of paths in the tree starting at the root. We simply associate to the endpoint of each path the permutation appearing in that vertex. We may further identify each path in the tree with the list of labels appearing on this path, but this requires to distinguish the potential repeated labels in the succession rule. More precisely, for each label $k\in\mathcal{L}$ and for each element $k'\in\CL(k)$ such that $\mult_{k}(k')>1,$  we distinguish the repeated labels by painting them  with different colors (say colors $\{1,\dots,\mult_{k}(k')\}$). Then the \emph{colored succession rule} is represented as 
\begin{equation}
	\begin{cases} 
	\label{eq:colored_succ_rule}
	\text{Root label}: (\lambda) \\
	(k)\to (e^{c}_1{\scriptstyle(k)}),\dots,(e^{c}_{h(k)}{\scriptstyle(k)}), 
	\end{cases}
\end{equation}	
where the exponents $c$ recall that the labels might be colored. We highlight that the colors and the values of the children labels depend only on the value of the parent label (and not on the color).
	
In this way, every permutation of size $n$ is bijectively encoded by a sequence of $n$ (colored) labels $(k_1=\lambda,k_2^{c_2},\dots,k_n^{c_n})$, where $k_i$ records the value of the label and $c_i$ the color,  such that every pair of two consecutive labels, $(k_i^{c_i},k_{i+1}^{c_{i+1}}),$ $1\leq i\leq n-1,$ is \emph{consistent with the colored succession rule}, i.e.\ for all $1\leq i\leq n-1,$ there exist $j=j(i)\in [1,h(k_i)]$ such that $k_{i+1}^{c_{i+1}}=e_{j}^{c}(k_i)$. We denote by $\GG$ this bijection between sequences of colored labels and permutations in the generating tree.

In order to simplify the notation, we will often write a sequence of colored labels as $(k_i^{c})_i$ instead of $(k_i^{c_i})_i$.

\begin{exmp}
	\label{exmp:1423_4123_av_part2}
	We continue Examples \ref{exmp:1423_4123_av_part0} and \ref{exmp:1423_4123_av}. In the succession rule given in \cref{eq:succ_rule_1423_4123} the label $(k+1)$ is produced twice. We distinguish the two occurrences of $(k+1)$ by painting the first one in blue (and we will write $\textcolor{blue}{{(k+1)}^{B}}$) and the second one in tangerine\footnote{The choice for the the colors is made with the aim of remembering that a Blue label activates the Bottom active site and a Tangerine label activates the Top active site.} (and we will write $\textcolor{orange}{{(k+1)}^{T}}$). The colored succession rule is 
	\begin{equation}\label{eq:biwfbwrubfwifnwepi}
	\begin{cases} 
	\text{Root label}: (2) \\
	(k)\to \textcolor{blue}{(k+1)^{B}},(3),(4)\dots,(k),\textcolor{orange}{(k+1)^{T}}. 
	\end{cases}
	\end{equation}
	
	With this coloring of the labels, the permutation 13254 (the only one displayed in the fifth level of the generating tree in \cref{gen_tree_1423_4123}) is encoded by the sequence of (colored) labels $(2,\textcolor{orange}{3^{T}},3,\textcolor{orange}{4^{T}},4),$ i.e.\ $\GG(2,\textcolor{orange}{3^{T}},3,\textcolor{orange}{4^{T}},4)=13254.$
\end{exmp}

\subsection{Main result}\label{sect:main_results}

In this section we formally state our main result. We first state three assumptions for a family of permutations $\mathcal{C}$ (growing on the right) encoded by a generating tree as above.

\subsubsection{The three assumptions}

\begin{ass}
	\label{ass1}
	There exists a $\Z$-valued statistic that determines (in the sense of \cref{def:gen_tree_part2}) the succession rule of the generating tree\footnote{These generating trees are usually called \emph{one-label generating trees}, and we used \emph{one-dimensional-labeled generating trees} earlier in the paper.}. We further require that the labels appearing in the generating tree are elements of $\mathbb{Z}_{\geq\beta}$ for some $\beta\in\mathbb{Z}$, except for the root, where the label is equal to $\lambda=\beta-1$. 
\end{ass}

\begin{rem}
	The first part of the assumption above might seem not completely mathematically 
	meaningful\footnote{For instance, taking any natural example of two-labels generating tree, it is possible to reinterpret it as a one-label generating tree, simply conjugating the $\Z^2$-valued statistic with a bijection 
	between $\Z^2$ and $\Z$. Note that the same small issue appears for instance counting the numbers of parameters in recurrence relations.}. Despite this, it is well established in the literature (see for instance \cite{MR2376115,bousquet2003four}) to refer to one-label generating trees, two-labels generating trees, etc.
	In addition, \cref{ass1} is completely mathematically meaningful when considered together with \cref{ass2} below.
\end{rem}

Let now $(\alpha_y)_{y\in\Z_{\leq 1}}$ be a probability distribution on $\Z_{\leq 1}$ and $(c_y)_{y\in\Z_{\leq 1}}$ be a sequence of non-negative integers such that $c_y\geq 1$ if and only if $\alpha_y>0$. Denote by ${\bm{Y}}^{\bm c}$ the corresponding colored $\mathbb{Z}_{\leq 1}$-valued random variable having value-distribution equal to $(\alpha_y)_{y\in\Z_{\leq 1}}$ (i.e.\ $\P({\bm{Y}}=y)=\alpha_y$ for all $y\in\Z_{\leq 1}$) and color-distribution defined as follows: conditioning on $\{{\bm{Y}}=y\}$ then $\bm{c}$ is a uniform element of $[c_y]$.

Consider now a sequence $({\bm{Y}}^{\bm c_j}_j)_{j\geq 1}$ of i.i.d.\ copies of ${\bm{Y}}^{\bm c}$.
We define the corresponding random colored walk as $({\bm{X}}_i^{\bm c_{i-1}})_{i\geq 1}$, where
\begin{equation}\label{eq:definition_walk}
{\bm{X}}_i\coloneqq (\beta-1)+\sum_{j=1}^{i-1}{\bm{Y}}_j,\quad\text{for all}\quad i\in\Z_{\geq 1}.
\end{equation}
Note that for all $i\in \Z_{\geq 2}$, ${\bm{X}}_i$ has color $\bm c_{i-1}$. In order to simplify the notation we will often simply write $({\bm{X}}_i^{\overleftarrow{\bm c}})_{i\geq 1}$ for the walk $({\bm{X}}_i^{\bm c_{i-1}})_{i\geq 1}$.

\begin{ass}
	\label{ass2}
	There exists a centered probability distribution $(\alpha_y)_{y\in\Z_{\leq 1}}$ on $\Z_{\leq 1}$ with finite variance and a sequence of non-negative integers $(c_y)_{y\in\Z_{\leq 1}}$ such that the corresponding one-dimensional random colored walk $({\bm{X}}^{\overleftarrow{\bm c}}_i)_{i\geq 1}$ (defined above) satisfies
	\begin{equation}\label{eq:rewriting_chain}
	\GG^{-1}(\bm\sigma_n)\stackrel{d}{=}\left({\bm{X}}_{i}^{\overleftarrow{\bm c}}\middle|({\bm{X}}_{j})_{j\in[2,n]}\geq \beta,{\bm{X}}_{n+1}=\beta\right)_{i\in[n]},
	\end{equation}
	where $\bm\sigma_n$ is a uniform permutation in $\mathcal C$ of size $n$. 
\end{ass}

\begin{rem}
	Note that, as a consequence of \cref{ass2}, we have that $\CL(k)\subseteq(-\infty,k+1]$, for all $k\in\mathcal{L}$. This is not as restrictive as it might seem. Indeed, many $\mathbb{Z}$-valued generating trees (see for instance Section \ref{sect:examples_gen_tree_Ok}) are constructed using the statistic that tracks the number of active sites and so, in many cases, adding a new element of a permutation increases the value of this statistic at most by one.
\end{rem}	

\begin{rem}
	The assumption that the distribution $(\alpha_y)_{y\in\Z_{\leq 1}}$ on $\Z_{\leq 1}$ has finite variance is also not too restrictive. Indeed, to the best of our knowledge, we are not aware of a class of permutations coded by a one-dimensional walk with finite expectation but
	infinite variance.
\end{rem}

\begin{rem}
	In \cref{sect:sampling_uniformly} we will explain how to determine a possible distribution $(\alpha_y)_{y\in\Z_{\leq 1}}$ and a sequence of non-negative integers $(c_y)_{y\in\Z_{\leq 1}}$ such that \cref{ass2} holds.
\end{rem}

For the third assumption we need to introduce the following definitions. 
We denote by $K^c$ the set of all finite sequences of \emph{colored} labels starting at $\lambda$ and consistent with the colored succession rule in \cref{eq:colored_succ_rule}. Moreover, $K^c_n$ is the set of sequences in $K^c$ of length $n.$  

\begin{defn}
	\label{defn:jumpslabel}
	Given a sequence $(k_i^{c_i})_{i\in[n]}\in K^c_n$ of (colored) labels in a generating tree, the corresponding sequence of \emph{colored jumps} is the sequence $(y_i^{c_{i+1}})_{i\in[n-1]}$ where $y_i=k_{i+1}-k_i$. Note that the label $y_i$ inherits the same color as the label $k_{i+1}$ (if it is colored).
\end{defn}

\begin{defn}\label{defn:rebfvriwbfowenf}
	A sequence of $h$ colored jumps $(y^{c_1}_1,\dots,y^{c_h}_h)$ is \emph{consistent} if it is equal to a factor of a sequence of jumps corresponding to an element of $K^c$.
\end{defn}

We denote by $\mathcal{Y}^c_h$ the set of consistent sequences of $h$ colored jumps. 

\begin{ass} \label{ass3} 
	For all $h \geq 1,$ there exists a function $\Pat:\mathcal{Y}^c_h \to\mathcal{S}^h $ and a constant $c(h )\geq 0$ such that if a sequence of labels $(k_i^{c_i})_{i\in[n]}\in K^c_n$ satisfies for some $m\in[n+1-h ]$ the condition 
	$$k_i>c(h ),\quad\text{for all}\quad i\in[m,m+h -1],$$ 
	then  $$\pat_{[m,m+h -1]}\big(\GG((k_i^{c_i})_{i\in[n]})\big)=\Pat\big((y_i^{c_{i+1}})_{i\in[m-1,m+h -2]}\big),$$
	where $(y_i^{c_{i+1}})_{i\in[n-1]}$ is the sequence of (colored) \emph{jumps} associated to $(k_i^{c_{i}})_{i\in[n]}\in K^c_n$.
\end{ass}
In words, the consecutive pattern induced on the interval $[m,m+h -1]$ in the permutation $\GG((k_i^{c_i})_{i\in[n]})$ is uniquely determined through the function $\Pat$ by the factor of colored jumps $(y_i^{c_{i+1}})_{i\in[m-1,m+h -2]}$ as long as the labels are large enough (see for instance \cref{exemp:messup} at page~\pageref{exemp:messup} to understand why this condition is necessary).

\medskip

Before stating our main result, we first furnish in the next section a strategy to check \cref{ass2} in many cases.

\subsubsection{Sampling uniform permutations encoded by generating trees as conditioned random walks}\label{sect:sampling_uniformly}


We fix a family of permutations $\mathcal{C}$ encoded by a generating tree that has $\mathbb{Z}$-valued labels. We further assume that the labels are elements of $\mathbb{Z}_{\geq\beta}$ for some $\beta\in\mathbb{Z}$, except for the root, where the label is equal to $\lambda=\beta-1$. 




We describe here a strategy to determine a distribution $(\alpha_y)_{y\in\Z_{\leq 1}}$ and a sequence of non-negative integers $(c_y)_{y\in\Z_{\leq 1}}$ such that the equality in \cref{eq:rewriting_chain} holds. 

\medskip

Recall that $\mathcal{Y}^c_1$ denotes the set of consistent colored jumps (see \cref{defn:rebfvriwbfowenf}). Denote by $\mathcal{Y}_1$ its uncolored version. Assume that the colored succession rule of the generating tree encoding the family $\mathcal{C}$ (see \cref{eq:colored_succ_rule}) is such that for all $k\in\Z_{\geq \beta-1},$

\begin{equation}\label{eq:ryuhfweubfweodq}
\Big\{(e^{c}_1{\scriptstyle(k)}),\dots,(e^{c}_{h(k)}{\scriptstyle(k)})\Big\}=\Big\{(k+y)^c:y^c\in\mathcal{Y}^c_1, k+y\geq \beta\Big\}.
\end{equation}
In particular, this implies that for all $k,k+y\in\Z_{\geq \beta-1}^2$, provided that $k\geq \beta-y$, the multiplicity $\mult_{k}(k+y)$ is independent of $k$ (note that if $k< \beta-y$ then $\mult_{k}(k+y)=0$). We set $m_y\coloneqq\mult_{k}(k+y)$ for some $k\geq \beta-y$. Note also that $m_y\neq 0$ if and only if $y\in\mathcal{Y}_1$.

Consider now (if they exist) $p,q>0$ such that
\begin{equation}\label{eq:frhweibdowenpen}
p\cdot\sum_{y\in\mathcal{Y}_1}q^y\cdot m_y=1,\qquad \sum_{y\in\mathcal{Y}_1}y\cdot q^y\cdot m_y=0\quad\text{and}\quad \sum_{y\in\mathcal{Y}_1}y^2\cdot q^y\cdot m_y<\infty.
\end{equation}

The existence of these parameters $p,q>0$ was already investigated by Janson~\cite{janson2012simply} in the context of simply generated trees with the following result.

\begin{lem}[Lemmas 4.1 and 4.2 in \cite{janson2012simply}]
	Set $\varphi(t)\coloneqq\sum_{y\in\mathcal{Y}_1}m_y\cdot t^{-y}=\sum_{y=-1}^\infty m_{-y}\cdot t^{y}.$ If there exists $t>0$ such that $0<\varphi(t)<\infty$, then there exists two parameters $p,q>0$ such that $p\cdot\sum_{y\in\mathcal{Y}_1}q^y\cdot m_y=1$. In particular, $p$ and $q$ are given by
	$$p\coloneqq\frac{1}{\varphi(t)}\quad \text{and} \quad q\coloneqq \frac 1 t.$$ Moreover, if $\bm \xi$ is the random variable with distribution $(p\cdot q^y\cdot m_y)_{y\in\Z_{\leq 1}}$ then
	$$\E[\bm \xi]=-\frac{t\varphi'(t)}{\varphi(t)}\eqqcolon -\psi(t)\quad\text{and}\quad \Var(\bm \xi)=t\psi'(t),$$
	and both $\E[\bm \xi]$ and $\Var(\bm \xi)$ are finite.
\end{lem}

We have the following consequence.

\begin{cor}
	Set $\varphi(t)\coloneqq\sum_{y\in\mathcal{Y}_1}m_y\cdot t^{-y}=\sum_{y=-1}^\infty m_{-y}\cdot t^{y}.$ Assume that there exists $t>0$ such that  $0<\varphi(t)<\infty$ and $\varphi'(t)=0$. Then there exist $p,q>0$ that satisfy \cref{eq:frhweibdowenpen}.
\end{cor}

Set, for all $y\in\Z_{\leq 1}$, $\xi_y\coloneqq p\cdot q^y\cdot m_y$. Then, from \cref{eq:frhweibdowenpen}, $(\xi_y)_{y\in\Z_{\leq 1}}$ is a centered probability distribution on $\Z_{\leq 1}$ with finite variance.

\begin{prop}\label{prop:ehibvreibcpiwrnfc}
	Under the conditions given in \cref{eq:ryuhfweubfweodq,eq:frhweibdowenpen}, \cref{ass2} is satisfied with  $\alpha_y=\xi_y$ and $c_y=m_y$, for all $y\in\Z_{\leq 1}$.
\end{prop}

\begin{proof}
	Let ${\bm{X}^{\overleftarrow{\bm c}}}$ be the random colored walk with jump distribution corresponding to $(\alpha_y)_{y\in\Z_{\leq 1}}$ and color distribution corresponding to  $(c_y)_{y\in\Z_{\leq 1}}$, starting at $\beta-1$ at time 1.
	Fix a sequence of consistent colored labels $(k^c_i)_{i\in[n]}$.
	Note that 
	\begin{multline*}
	\P\left(({\bm{X}}^{\overleftarrow{\bm c}}_{i})_{i\in[n]}=(k^c_i)_{i\in[n]},{\bm{X}}_{n+1}=\beta\right)=\P\left(({\bm{X}}^{\overleftarrow{\bm c}}_{i})_{i\in[n]}=(k^c_i)_{i\in[n]}\right)\cdot \P\left({\bm{X}}_{n+1}=\beta|{\bm{X}}_{n}=k_n\right)\\
	=\left(p^{n-1} \cdot q^{k_n-\beta+1}\right)\cdot\left(p \cdot q^{\beta-k_n} \right)
	=p^n\cdot q.
	\end{multline*} 
	Since the  probability above is independent of $(k^c_i)_{i\in[n]}$, and every possible path of the walk $({\bm{X}}^{\overleftarrow{\bm c}}_{i}|({\bm{X}}_{j})_{j\in[2,n]}\geq \beta,{\bm{X}}_{n+1}=\beta)_{i\in[n]}$ corresponds to exactly one path in the generating tree starting from the root (thanks to \cref{eq:ryuhfweubfweodq}), we can conclude that 
	\begin{equation*}
	\GG^{-1}(\bm\sigma_n)\stackrel{d}{=}\left({\bm{X}}^{\overleftarrow{\bm c}}_{i}\middle|({\bm{X}}_{j})_{j\in[2,n]}\geq \beta,{\bm{X}}_{n+1}=\beta\right)_{i\in[n]},
	\end{equation*}
	where $\bm\sigma_n$ is a uniform permutation in $\mathcal{C}$ of size $n$. 
\end{proof}

\subsubsection{The statement of the main result}

We can now state our main result.

\begin{thm}
	\label{thm:main_thm_CLT}
	Let $\mathcal{C}$ be a family of permutations that satisfies Assumptions \ref{ass1}, \ref{ass2} and \ref{ass3}. Let $({\bm{Y}}^{\bm c_i}_i)_{i\geq1}$ be a sequence of i.i.d.\ random variables distributed as ${\bm{Y}}^{\bm c}.$ Let $\bm{\sigma}_n$ be a uniform random permutation in $\mathcal{C}^n,$ for all $n\in\Z_{>0}.$
	Then, for all $\pi\in\mathcal{S},$ we have the following central limit theorem,
	\begin{equation}
	\label{wfowrbfv3}
	\frac{\coc(\pi,\bm{\sigma}_n)-n\mu_{\pi}}{\sqrt n}\stackrel{d}{\longrightarrow}\bm{\mathcal N}(0,\gamma_{\pi}^2),
	\end{equation}
	where $\mu_{\pi}=\P\left(\Pat\left({\bm{Y}}_1^{\bm c},\dots,{\bm{Y}}^{\bm c}_{|\pi|}\right)=\pi\right)$ and $\gamma_{\pi}^2=\beta^2-\frac{\rho^2}{\sigma^2}$ with
	\begin{align*}
	&\sigma^2=\Var\left({\bm{Y}}_1\right),\\
	&\rho=\E\left[\mathds{1}_{\left\{\Pat\left({\bm{Y}}_1^{\bm c},\dots,{\bm{Y}}^{\bm c}_{|\pi|}\right)=\pi\right\}}\cdot \sum_{j=1}^{|\pi|}\bm{Y}_j\right],\\
	&\beta^2=2\nu+\mu_\pi-(2|\pi|-1)\cdot\mu_\pi^2,\quad\text{for}\\
	&\nu=\sum_{s=2}^{|\pi|}\sum_{\substack{
			(y^{c}_i)_{i\in[|\pi|]},(\ell^{d}_i)_{i\in[|\pi|]}\in\Pat^{-1}(\pi)\\
			\text{s.t. } ({y}^{c}_s,\dots,{y}^{c}_{|\pi|})=({\ell}^{c}_1,\dots,{\ell}^{d}_{|\pi|-s+1})}}
	\P\left(
	({\bm{Y}}^{\bm c}_{i})_{i\in[|\pi|+s-1]}=({y}^c_1,\dots,{y}^c_{|\pi|},{\ell}^d_{|\pi|-s+2},\dots,{\ell}^d_{|\pi|})\right).
	\end{align*}
\end{thm}

The proof of this theorem can be found in Section \ref{sect:main_thm_proof}. A consequence of the above result is the local limit result in \cref{corl:main_thm} below. We recall that a detailed presentation of the local topology for permutations is given in \cref{sect:appe_local_topo}. In particular see Definitions \ref{defn:weakweakconv} and \ref{strongconv} for explanations on the annealed and quenched Benjamini--Schramm convergence.

\begin{cor}\label{corl:main_thm}
	Let $\mathcal{C}$ be a family of permutations encoded by a generating tree that satisfies Assumptions \ref{ass1}, \ref{ass2} and \ref{ass3}. There exists a random infinite rooted permutation $\bm{\sigma}^\infty_{\mathcal{C}}$ such that  
	$$\mathcal{L}aw\big((\bm{\sigma}_n,\bm{i}_n)\big|\bm{\sigma}_n\big)\stackrel{d}{\to}\mathcal{L}aw(\bm{\sigma}^\infty_{\mathcal{C}})
	\quad\text{and}\quad
	(\bm{\sigma}_n,\bm{i}_n)\stackrel{d}{\to}\bm{\sigma}^\infty_{\mathcal{C}},$$
	i.e.\ $\bm{\sigma}_n$ converges both in the quenched and annealed sense.	
\end{cor}

\begin{rem}
	Note that \cref{corl:main_thm} states the existence of a random infinite rooted permutation $\bm{\sigma}^\infty_\mathcal{C}$ without explicitly constructing this limiting object. Nevertheless, we provide an explicit construction of $\bm{\sigma}^\infty_\mathcal{C}$ in Section \ref{sec:construction}. 
\end{rem}

\subsection{Future projects and open questions}\label{sect:fututre projects}

We collect some ideas that we would like to investigate in future projects, together with some open questions.

\begin{itemize}
	\item A natural question related to the result stated in \cref{thm:main_thm_CLT} is the following: is it possible to prove joint convergence in distribution of consecutive pattern densities to a Gaussian vector? We believe that the fundamental step to reach this goal is to prove a ``multi-dimensional" version of \cref{lem:lemma3}. In our opinion this should not present great difficulties (except for the presence of a more complicated notation and an higher complexity of the various formulas/expressions involved in the proof).

	\item We believe that our Assumptions \ref{ass1} and \ref{ass2}, in particular the fact that the generating tree has one-dimensional labels, can be relaxed. Our assumptions reduce the problem on permutations to the study of one-dimensional random walks conditioned to be positive. Relaxing Assumptions \ref{ass1} and \ref{ass2} to multi-dimensional labels, leads to the study of \emph{multi-dimensional random walks conditioned to stay in a cone}. We believe that the work of Denisov and Wachtel~\cite{MR3342657} furnishs some useful results to investigate this problem, specifically to prove an analogue of the results proved in \cref{sect:CTL_increments} for random walks in cones.
	
	A motivation for studying this more general framework is given by that fact that several families of permutations have been enumerated using multi-dimensional generating trees, for instance the famous case of Baxter permutations (see \cite{bousquet2003four}).
	
	\item It is not evident to us that the expressions of the variances $\gamma_{\pi}^2$ in \cref{thm:main_thm_CLT} satisfy $\gamma_{\pi}^2>0$. It would be interesting to investigate this problem. Moreover, given a specific family of permutation $\mathcal{C}$, it would be interesting to explicitly compute the expressions for $\mu_\pi$, for all $\pi\in\mathcal{C}$, as done for instance for the families $\Av(123)$ and $\Av(132)$ in \cite[Theorem 1.8,Theorem 1.9]{borga2018local}. We highlight that the coefficients $\mu_\pi$ satisfies several relations (for instance the trivial relations $0\leq\mu_\pi\leq 1$ and $\sum_{\pi\in\mathcal{S}^k}\mu_\pi=1$) recently described in \cite{borga2019feasible,borga2020feasible}.

	\item In this paper we use generating trees to establish central limit theorems and local limit results. We believe that they can be also used as a nice tool for establishing scaling limits (in the permuton sense) of random permutations. Together with Mickaël Maazoun, we explored in \cite{Baxterscallim} the permuton limit of Baxter permutations. In this work we used as a key tool a bijection between Baxter permutations and walks in cones that was already available in the literature (and that does not come from generating trees). Nevertheless, we believe that the techniques developed in \cite{Baxterscallim} apply also to other family of permutations where a bijection with walks in cones is available, and for finding such bijections, generating trees might be really useful.
\end{itemize}

\section{Proof of the main result}
\label{sec:locthm}

This section is organized as follows:

\begin{itemize}
	\item In \cref{sect:CTL_increments} we collect all the results on conditioned random walks that we need for the proof of our main result.
	\item In \cref{sect:main_thm_proof} we prove \cref{thm:main_thm_CLT} and \cref{corl:main_thm}.
	\item Finally, in \cref{sec:construction} we exhibit an explicit construction of the limiting object $\bm{\sigma}^{\infty}_{\mathcal{C}}$ (appearing in \cref{corl:main_thm}) as a random total order $\bm{\preccurlyeq}_{\mathcal{C}}$ of $\mathbb{Z}.$
\end{itemize} 

\subsection{Results on conditioned random walks}\label{sect:CTL_increments}

We recall that $({\bm{X}}^{\overleftarrow{\bm c}}_{i})_{i\in\Z_{\geq 1}}=({\bm{X}}^{\bm c_{i-1}}_{i})_{i\in\Z_{\geq 1}}$ denotes the random colored walk defined in \cref{eq:definition_walk}. We assume that the distribution of its increments $(\alpha_y)_{y\in\Z_{\leq 1}}$ is centered with finite variance and that ${\bm X}_1=-1$ (i.e. $\beta=0$ in \cref{eq:definition_walk}). We also recall that we denote the colored increments of $({\bm{X}}^{\bm c_{i-1}}_{i})_{i\in\Z_{\geq 1}}$ by ${\bm{Y}}^{\bm c_i}_{i}$ in such a way that $\bm{Y}_i={\bm{X}}_{i+1}-{\bm{X}}_{i},$ for all $i\in\Z_{\geq 1}.$

For simplicity we also assume that the span of the distribution $(\alpha_y)_{y\in\Z_{\leq 1}}$ is 1 (setting $y'=\max\{y\in\Z_{\leq 1}:\alpha_{y}> 0\}$, the \emph{span} of $(\alpha_y)_{y\in\Z_{\leq 1}}$ is the 
greatest common divisor of the set $\{y'-y:y\in{\Z_{< y'}},\alpha_{y}> 0\}$), omitting the minor (and standard) modifications in the general case. For instance, if the span is $h$, when using local limits results (as in \cref{eq:locallimthm}), one has to replace the expression  
\begin{equation*}
\P\left({\bm{X}}_{n}=\ell\right)=\frac{1}{\sqrt{2 \pi \sigma^2 n}}\left(e^{-\tfrac{\ell^2}{2n\sigma^2}}+o(1)\right),\quad\text{for all $\ell\in\Z$,}
\end{equation*}
with
\begin{equation*}
\P\left({\bm{X}}_{n}=\ell\right)=\frac{h}{\sqrt{2 \pi \sigma^2 n}}\left(e^{-\tfrac{\ell^2}{2n\sigma^2}}+o(1)\right),\quad\text{for all $\ell\in h\Z-1$}.
\end{equation*}

\subsubsection{Random walks conditioned to stay positive are larger than any fixed constant}\label{sect:CTL_increments1}

We state our first result on conditioned random walks, which will be helpful in the proof of \cref{thm:main_thm_CLT} in connection with \cref{ass3}.

\begin{prop}
	\label{prop:prop_labels_big}
	Let $c>0$ be a constant. Let $(a_n)_{n\in\Z_{\geq 0}}$ be a sequence of integers such that $a_n<n/2$ for all $n\in\Z_{\geq 0}$ and $a_n\to \infty$. Then, as $n$ tends to infinity,
	\begin{equation}\label{eq:goal_proof}
	\P\left({\bm{X}}_{i}>c \text{ for all } i\in[a_n,n-a_n]\Big|({\bm{X}}_{j})_{j\in[2,n]}\geq 0,{\bm{X}}_{n+1}=0\right)\to 1.
	\end{equation}
\end{prop}

\begin{proof} 
	Note that, using a union bound, we have
	\begin{multline}\label{eq:step1}
	\P\left({\bm{X}}_{i}>c\text{ for all } i\in[a_n,n-a_n]\Big|({\bm{X}}_{j})_{j\in[2,n]}\geq 0,{\bm{X}}_{n+1}=0\right)\\
	\geq 1-\sum_{i\in[a_n,n-a_n]}\P\left({\bm{X}}_{i}\leq c\Big|({\bm{X}}_{j})_{j\in[2,n]}\geq 0,{\bm{X}}_{n+1}=0\right).
	\end{multline}
	Fix now $i\in[a_n,\lfloor n/2\rfloor]$.
	Denoting with $\hat{\bm{X}}$ the time-reversed walk of ${\bm{X}}$ starting at 0 at time 1, we have that
	\begin{multline*}
	\P\left({\bm{X}}_{i}\leq c\Big|({\bm{X}}_{j})_{j\in[2,n]}\geq 0,{\bm{X}}_{n+1}=0\right)=
	\sum_{y=0}^c\frac{\P\left({\bm{X}}_{i}=y,{\bm{X}}_{n+1}=0,({\bm{X}}_{j})_{j\in[2,n]}\geq 0\right)}{\P\left({\bm{X}}_{n+1}=0,({\bm{X}}_{j})_{j\in[2,n]}\geq 0\right)}\\
	=\sum_{y=0}^c\frac{\P\left({\bm{X}}_{i}=y,({\bm{X}}_{j})_{j\in[2,i]}\geq 0\right)\P\left(\hat{\bm{X}}_{n+1-i}=y,(\hat{\bm{X}}_{j})_{j\in[n+1-i]}\geq 0\right)}{\P\left({\bm{X}}_{n+1}=0,({\bm{X}}_{j})_{j\in[2,n]}\geq 0\right)}.
	\end{multline*}
	Using a local limit theorem for random walks conditioned to stay positive (see \cite[Theorem 3]{MR2449127}) we have that
	\begin{align*}
	&\P\left({\bm{X}}_{i}=y,({\bm{X}}_{j})_{j\in[2,i]}\geq 0\right)\sim C_1(y)\cdot i^{-3/2},\\
	&\P\left(\hat{\bm{X}}_{n+1-i}=y,(\hat{\bm{X}}_{j})_{j\in[n+1-i]}\geq 0\right)\sim C_2(y)\cdot(n-i)^{-3/2},\\
	&\P\left({\bm{X}}_{n+1}=0,({\bm{X}}_{j})_{j\in[2,n]}\geq 0\right)\sim C_3\cdot n^{-3/2}.
	\end{align*}
	Since $i\in[a_n,\lfloor n/2\rfloor]$, $a_n\to\infty,$ and $y\in[0,c]\cap \Z$, we obtain that
	\begin{equation*}
	\P\left({\bm{X}}_{i}\leq c\Big|({\bm{X}}_{j})_{j\in[2,n]}\geq 0,{\bm{X}}_{n+1}=0\right)=O(i^{-3/2}),
	\end{equation*}
	uniformly for all $i\in[a_n,\lfloor n/2\rfloor]$.
	This bound together with \cref{eq:step1} leads to
	\begin{equation*}
	\P\left({\bm{X}}_{i}>c\text{ for all } i\in[a_n,\lfloor n/2\rfloor]\Big|({\bm{X}}_{j})_{j\in[2,n]}\geq 0,{\bm{X}}_{n+1}=0\right)\geq 1-O((a_n)^{-1/2}).
	\end{equation*}
	Using the same techniques, we have the same bound for $i\in[\lfloor n/2\rfloor,n-a_n]$, and we can conclude the proof.
\end{proof}

\subsubsection{A CLT for consecutive increments in random walks conditioned to stay positive}\label{sect:CTL_increments2}

In this section we prove the following CLT for occurrences of patterns of consecutive colored increments in random walks conditioned to stay positive.

\begin{prop}\label{lem:lemma3}
	Fix $h\in\Z_{>0}$. Let ${\mathcal A}$ be a (possibly infinite) family of sequences of $h$ colored increments for the walk $({\bm{X}}^{\overleftarrow{\bm c}}_i)_{i\geq 1}$.  Then
	\begin{equation*}
	\left(\frac{\sum_{j\in[1,n-h+1]}\mathds{1}_{\left\{({\bm{Y}}^{\bm c_i}_i)_{i\in[j,j+h-1]}\in{\mathcal A}\right\}}-n\cdot \mu_{{\mathcal A}}}{\sqrt n}\Bigg|({\bm{X}}_i)_{i\in[2,n]}\geq 0,{\bm{X}}_{n+1}=0\right)\stackrel{d}{\longrightarrow}\bm{\mathcal{N}}(0,\gamma_{{\mathcal A}}^2),
	\end{equation*}
	where $\mu_{{\mathcal A}}=\P\big(({\bm{Y}}^{\bm c_1}_1,\dots,{\bm{Y}}^{\bm c_h}_{h})\in {\mathcal A}\big)$ and $\gamma_{{\mathcal A}}^2=\beta^2-\frac{\rho^2}{\sigma^2}$ with
	\begin{align*}
	&\sigma^2=\Var({\bm{Y}}_1),\\
	&\rho=\E\left[\mathds{1}_{\left\{({\bm{Y}}^{\bm c_i}_{i})_{i\in[h]}\in{\mathcal A}\right\}}\cdot\left({\bm{X}}_{h+1}+1\right)\right],\\
	&\beta^2=2\nu+\mu_{{\mathcal A}}-(2h-1)\mu_{{\mathcal A}}^2,\quad\text{for}\\
	&\nu=\sum_{s=2}^h\sum_{\substack{
			({y}^c_i)_{i\in[h]},({\ell}^d_i)_{i\in[h]}\in{\mathcal A}\text{ s.t.}\\
			({y}^c_s,\dots,{y}^c_h)=({\ell}^d_1,\dots,{\ell}^d_{h-s+1})}}
	\P\left(({\bm{Y}}^{\bm c}_{i})_{i\in[h+s-1]}=({y}^c_1,\dots,{y}^c_h,{\ell}^d_{h-s+2},\dots,{\ell}^d_h)\right).
	\end{align*}
\end{prop}

The proof of the result above is inspired by the proof of \cite[Lemma 7.1]{MR3432572} that was, in turn, based on a method introduced by Le Cam \cite{MR105735} and Holst \cite{MR628875}. Since our conditions are different, we include a complete proof of our result.

In the proof we denote by $\bm O(1)$ an unspecified sequence of random variables $(\bm \varepsilon_n)_{n\in\Z_{>0}}$ that are bounded by a constant (i.e.\ there exists a constant $C>0$ such that $\bm \varepsilon_n\leq C$ a.s.\ for all $n\in\Z_{>0}$).

\begin{proof}[Proof of \cref{lem:lemma3}]
	Recall that $\hat{\bm X}$ denotes the time-reversed walk of $\bm X$ starting at 0 at time 1. Denote by $\hat{\mathcal A}$ the family obtained from ${\mathcal A}$ by reversing the time and changing the sign of each sequence, i.e.\ if $(y^{c_i}_i)_{i\in [h]}\in\mathcal{A}$ then $(-y^{c_{h-i+1}}_{h-i+1})_{i\in [h]}\in\hat{\mathcal{A}}$. Note that, the increments $(\hat{\bm{Y}}_{i})_{i\geq 1}$ of the walk $\hat{\bm{X}}$ are supported on $\Z_{\geq -1}$ and so using the cycle lemma (see for instance \cite[Lemma 6.1]{MR2245368}) we can cyclically shift the increments of the walk in order to
	ensure that $(\hat{\bm{X}}_i)_{i\in[n]}\geq 0$. We therefore obtain that
	\begin{multline*}
	\left(\sum_{j\in[1,n-h+1]}\mathds{1}_{\left\{(\hat{\bm{Y}}^{\bm c_i}_i)_{i\in[j,j+h-1]}\in\hat{\mathcal A}\right\}}\Bigg|(\hat{\bm{X}}_i)_{i\in[n]}\geq 0,\hat{\bm{X}}_{n+1}=-1\right)\\
	\stackrel{d}{=}\left(\sum_{j\in[1,n-h+1]}\mathds{1}_{\left\{(\hat{\bm{Y}}^{\bm c_i}_i)_{i\in[j,j+h-1]}\in\hat{\mathcal A}\right\}}\Bigg|\hat{\bm{X}}_{n+1}=-1\right)+\bm{O}(1),
	\end{multline*}
	where the error term is due to the fact that the cyclic shift can only effect $\bm{O}(1)$ of the consecutive occurrences (because of possible boundary problems). More precisely, a.s.\ $\bm{O}(1)\leq 2h$.
	As a consequence,
	\begin{multline*}
	\left(\sum_{j\in[1,n-h+1]}\mathds{1}_{\left\{({\bm{Y}}^{\bm c_i}_i)_{i\in[j,j+h-1]}\in{\mathcal A}\right\}}\Bigg|({\bm{X}}_i)_{i\in[2,n]}\geq 0,{\bm{X}}_{n+1}=0\right)\\
	\stackrel{d}{=}\left(\sum_{j\in[1,n-h+1]}\mathds{1}_{\left\{({\bm{Y}}^{\bm c_i}_i)_{i\in[j,j+h-1]}\in{\mathcal A}\right\}}\Bigg|{\bm{X}}_{n+1}=0\right)+\bm{O}(1).
	\end{multline*}
	In words, the above equality implies that we can forget the conditional event $\{({\bm{X}}_i)_{i\in[2,n]}\geq 0\}$ in our analysis.
	
	 We set 
	\begin{equation*}
	g({\bm{Y}}^{\bm c}_j,\dots, {\bm{Y}}^{\bm c}_{j+h-1})\coloneqq\mathds{1}_{\left\{({\bm{Y}}^{\bm c}_{j+i})_{i\in[0,h-1]}\in{\mathcal A}\right\}}.
	\end{equation*}	
	Define the centered sum
	\begin{equation*}
	\bm{S}_n\coloneqq\sum_{j=1}^n\left(g({\bm{Y}}^{\bm c}_j,\dots, {\bm{Y}}^{\bm c}_{j+h-1})-\mu_{{\mathcal A}}\right),
	\end{equation*}
	with the convention that $g({\bm{Y}}^{\bm c}_j,\dots, {\bm{Y}}^{\bm c}_{j+h-1})=0$ if $j>n-h+1$.
	
	We fix $\alpha$ with $0<\alpha<1$ and a sequence $n'=n'(n)$ with $\frac{n'}{n}\to\alpha$, for instance $n'=\lfloor \alpha n\rfloor$.
	By the central limit theorem for $h$-dependent variables (see \cite{MR26771,MR67396}), applied to the random vectors 
	$$\left(g({\bm{Y}}^{\bm c}_j,\dots, {\bm{Y}}^{\bm c}_{j+h-1})-\mu_{{\mathcal A}},{\bm{Y}}_j\right),$$ we have the following unconditioned result
	\begin{equation}\label{eq:clt_uncond}
	\left(\frac{\bm{S}_{n'}}{\sqrt{n}},\frac{{\bm{X}}_{n'}}{\sqrt n}\right)\xrightarrow{d}\bm{\mathcal{N}}\left(0,\alpha\left(\begin{array}{cc}
	\beta^2 & \rho\\
	\rho & \sigma^2
	\end{array}\right)\right),
	\end{equation}
	where
	\begin{align*}
	&\sigma^2=\Var({\bm{Y}}_1),\\
	&\beta^2=\Var\left(g({\bm{Y}}^{\bm c}_1,\dots, {\bm{Y}}^{\bm c}_{h})\right)+2\sum_{s=2}^h\Cov\left(g({\bm{Y}}^{\bm c}_1,\dots, {\bm{Y}}^{\bm c}_{h}),g({\bm{Y}}^{\bm c}_s,\dots, {\bm{Y}}^{\bm c}_{s+h-1})\right),\\
	&\rho=\sum_{s=1}^h\Cov\left(g({\bm{Y}}^{\bm c}_1,\dots, {\bm{Y}}^{\bm c}_{h}), {\bm{Y}}_{s}\right)=\Cov\left(g({\bm{Y}}^{\bm c}_1,\dots, {\bm{Y}}^{\bm c}_{h}), {\bm{X}}_{h+1}+1\right).
	\end{align*}
	
	We first compute $\beta^2$. Note that
	\begin{equation*}
	\Var\left(g({\bm{Y}}^{\bm c}_1,\dots, {\bm{Y}}^{\bm c}_{h})\right)=\E\left[\left(\mathds{1}_{\left\{({\bm{Y}}^{\bm c}_{i})_{i\in[h]}\in{\mathcal A}\right\}}\right)^2\right]-\mu_{{\mathcal A}}^2=\mu_{{\mathcal A}}-\mu_{{\mathcal A}}^2.
	\end{equation*}
	Now fix $2\leq s\leq h,$
	\begin{equation*}
	\Cov\left(g({\bm{Y}}^{\bm c}_1,\dots, {\bm{Y}}^{\bm c}_{h}),g({\bm{Y}}^{\bm c}_s,\dots, {\bm{Y}}^{\bm c}_{s+h-1})\right)=\E\left[\left(\mathds{1}_{\left\{({\bm{Y}}^{\bm c}_{i})_{i\in[h]}\in{\mathcal A}\right\}}\right)\left(\mathds{1}_{\left\{({\bm{Y}}^{\bm c}_{s+i-1})_{i\in[h]}\in{\mathcal A}\right\}}\right)\right]-\mu_{{\mathcal A}}^2.
	\end{equation*}
	Noting that, for $({y}^{c_i}_i)_{i\in[h]},({\ell}^{d_i}_i)_{i\in[h]}\in{\mathcal{A}}$, then $\mathds{1}_{\left\{({\bm{Y}}^{\bm c}_{i})_{i\in[h]}=({y}^c_i)_{i\in[h]},({\bm{Y}}^{\bm c}_{s+i-1})_{i\in[h]}=({\ell}^d_i)_{i\in[h]}\right\}}\neq 0$ only if $({y}^c_s,\dots,{y}^c_h)=({\ell}^d_1,\dots,{\ell}^d_{h-s+1})$, we have
	\begin{multline*}
	\Cov\left(g({\bm{Y}}^{\bm c}_1,\dots, {\bm{Y}}^{\bm c}_{h}),g({\bm{Y}}^{\bm c}_s,\dots, {\bm{Y}}^{\bm c}_{s+h-1})\right)\\
	=\E\left[\sum_{\substack{
			({y}^c_i)_{i\in[h]}, ({\ell}^d_i)_{i\in[h]}\in{\mathcal A}\text{ s.t.}\\
			({y}^c_s,\dots,{y}^c_h)=({\ell}^d_1,\dots,{\ell}^d_{h-s+1})}}
	\mathds{1}_{\left\{({\bm{Y}}^{\bm c}_{i})_{i\in[h+s-1]}=({y}^c_1,\dots,{y}^c_h,{\ell}^d_{h-s+2},\dots,{\ell}^d_h)\right\}}\right]-\mu_{{\mathcal A}}^2,
	\end{multline*}
	and we obtain that $\beta^2=2\nu+\mu_{{\mathcal{A}}}-(2h-1)\mu_{{\mathcal{A}}}^2$, where $\nu$ is defined in the statement of the theorem.
	
	We finally compute $\rho$. Note that since $\E[{\bm{X}}_{h+1}+1]=0$ (see \cref{eq:definition_walk} and recall that $\beta=0$) then
	\begin{equation*}
	\Cov\left(g({\bm{Y}}^{\bm c}_1,\dots, {\bm{Y}}^{\bm c}_{h}), {\bm{X}}_{h+1}+1\right)=\E\left[\mathds{1}_{\left\{({\bm{Y}}^{\bm c}_{i})_{i\in[h]}\in{\mathcal A}\right\}}\cdot\left({\bm{X}}_{h+1}+1\right)\right],
	\end{equation*}
	which corresponds to the expression in the statement of the proposition.
	
	We now define for convenience
	\begin{equation*}
	\bar{\bm{S}}_n\coloneqq\bm{S}_n-\frac{\rho}{\sigma^2}{\bm{X}}_{n}.
	\end{equation*}
	Since we know from \cref{eq:clt_uncond} that
	\begin{equation*}
	\left(\frac{\bm{S}_{n'}}{\sqrt{n}},\frac{{\bm{X}}_{n'}}{\sqrt n}\right)\xrightarrow{d}\bm{\mathcal{N}}\left(0,\alpha\left(\begin{array}{cc}
	\beta^2 & \rho\\
	\rho & \sigma^2
	\end{array}\right)\right),
	\end{equation*}
	then with some basic computations we obtain that
	\begin{equation*}
	\left(\frac{\bar{\bm{S}}_{n'}}{\sqrt{n}},\frac{{\bm{X}}_{n'}}{\sqrt n}\right)\xrightarrow{d}\bm{\mathcal{N}}\left(0,\alpha\left(\begin{array}{cc}
	\beta^2-\frac{\rho^2}{\sigma^2} & 0\\
	0 & \sigma^2
	\end{array}\right)\right).
	\end{equation*}
	In other words, $\frac{\bar{\bm{S}}_{n'}}{\sqrt{n}}$ and $\frac{{\bm{X}}_{n'}}{\sqrt n}$ are jointly asymptotically normal with independent limits $\bm W=\bm{\mathcal{N}}\left(0,\alpha\cdot\gamma_{{\mathcal A}}^2\right)$, where $\gamma_{{\mathcal A}}^2=\beta^2-\frac{\rho^2}{\sigma^2}$, and $\bm R=\bm{\mathcal{N}}(0,\alpha\cdot \sigma^2)$.
	
	Next, let $f$ be any bounded continuous function on $\mathbb R$. Then, using a local limit theorem for random walks (see for instance \cite[Theorem VII.1]{MR0388499}) we have that, uniformly for all $\ell\in\Z$,
	\begin{equation}\label{eq:locallimthm}
	\P\left({\bm{X}}_{n}=\ell\right)=\frac{1}{\sqrt{2 \pi \sigma^2 n}}\left(e^{-\tfrac{\ell^2}{2n\sigma^2}}+o(1)\right),
	\end{equation}
	and so 	
	\begin{equation*}
	\begin{split}
	\E\bigg[f\left(\frac{\bar{\bm{S}}_{n'}}{\sqrt{n}}\right)&\bigg|{\bm{X}}_{n+1}=0\bigg]\\
	&=\frac{\E\left[\sum_jf\left(\frac{\bar{\bm{S}}_{n'}}{\sqrt{n}}\right)\mathds{1}_{\{{\bm{X}}_{n'}=j\}}\mathds{1}_{\{{\bm{X}}_{n+1}-{\bm{X}}_{n'}=-j\}}\right]}{\P\left({\bm{X}}_{n+1}=0\right)}\\
	&=\frac{\sum_j\E\left[f\left(\frac{\bar{\bm{S}}_{n'}}{\sqrt{n}}\right)\mathds{1}_{\{{\bm{X}}_{n'}=j\}}\right]\P\left({\bm{X}}_{n-n'+1}=-j\right)}{\P\left({\bm{X}}_{n+1}=0\right)}\\
	&=\sum_j\E\left[f\left(\frac{\bar{\bm{S}}_{n'}}{\sqrt{n}}\right)\mathds{1}_{\{{\bm{X}}_{n'}=j\}}\right]\sqrt{\frac{n+1}{n-n'+1}}\left(e^{-\tfrac{j^2}{2(n-n')\sigma^2}}+o(1)\right)\\
	&=\sqrt{\frac{n+1}{n-n'+1}}\cdot\E\left[f\left(\frac{\bar{\bm{S}}_{n'}}{\sqrt{n}}\right)e^{-\tfrac{{\bm{X}}_{n'}^2}{2(n-n')\sigma^2}}\right]+o(1)\\
	&\to(1-\alpha)^{-1/2}\cdot\E\left[f(\bm W)e^{-\tfrac{\bm R^2}{2(1-\alpha)\sigma^2}}\right],
	\end{split}
	\end{equation*}
	where we used that $\frac{n'}{n}\to\alpha$, $\frac{\bar{\bm{S}}_{n'}}{\sqrt{n}}\xrightarrow{d}\bm W$, $\frac{{\bm{X}}_{n'}}{\sqrt n}\xrightarrow{d}\bm R$. Note that $$(1-\alpha)^{-1/2}\cdot\E\left[f(\bm W)e^{-\tfrac{\bm R^2}{2(1-\alpha)\sigma^2}}\right]=\E\left[f(\bm W)\right],$$ 
this easily follows taking $f=1$ (or from standard computations with Gaussian variables). The last two equations prove that 
	\begin{equation}\label{eq:reachedoneobj}
	\left(\frac{\bar{\bm{S}}_{n'}}{\sqrt{n}}\bigg|{\bm{X}}_{n+1}=0\right)\xrightarrow{d}\bm W=\bm{\mathcal{N}}\left(0,\alpha\cdot \gamma_{{\mathcal A}}^2\right).
	\end{equation}	
	
	We now look at
	\begin{equation*}
		\bm{S}_{n}-\bar{\bm{S}}_{n'}
		=\sum_{j=n'+1}^{n}\left(g({\bm{Y}}^{\bm c}_j,\dots, {\bm{Y}}^{\bm c}_{j+h-1})-\mu_{{\mathcal A}}\right)+\frac{\rho}{\sigma^2}{\bm{X}}_{n'}.
	\end{equation*}
	Recalling that $\hat{\bm{X}}$ denotes the time-reversed walk of ${\bm{X}}$ starting at 0 at time 1 (and $\hat{\bm Y}^{\bm c}_{j}$ the corresponding reversed steps), we have from the above equation that
	\begin{multline}\label{eq:bfirybfuiwbf0ow}
	\left(\bm{S}_{n}-\bar{\bm{S}}_{n'}\Big|{\bm{X}}_{n+1}=0\right)\\
	\stackrel{d}{=}\left(\sum_{j=1}^{n-n'}\left({g}(\hat{\bm{Y}}^{\bm c}_j,\dots, \hat{\bm Y}^{\bm c}_{j+h-1})-\mu_{{\mathcal A}}\right)+\frac{\rho}{\sigma^2}\hat{\bm{X}}_{n-n'+1}\Bigg|\hat{\bm{X}}_{n+1}=-1\right),
	\end{multline}
	where
	\begin{equation*}
	{g}(\hat{\bm{Y}}^{\bm c}_j,\dots, \hat{\bm Y}^{\bm c}_{j+h-1})\coloneqq\mathds{1}_{\left\{(\hat{\bm{Y}}^{\bm c}_{j+i})_{i\in[0,h-1]}\in{\hat{\mathcal A}}\right\}}.
	\end{equation*}	
	
	Note that, using \cref{eq:bfirybfuiwbf0ow} and similar arguments to the ones used for proving \cref{eq:reachedoneobj}, we obtain that
	\begin{equation}\label{eq:fiugwebhdndpwemd}
		\left(\frac{\bm{S}_{n}-\bar{\bm{S}}_{n'}}{\sqrt{n}}\bigg|{\bm{X}}_{n+1}=0\right)\xrightarrow{d}\bm{\mathcal{N}}\left(0,(1-\alpha)\gamma_{\hat{\mathcal A}}^2\right),
	\end{equation}
	for some constant $\gamma_{\hat{\mathcal A}}^2$.
	
	Letting $\alpha\to 1$, we obtain that
	\begin{align*}
	\bm{\mathcal{N}}\left(0,\alpha\cdot \gamma_{{\mathcal A}}^2\right)\xrightarrow{d}\bm{\mathcal{N}}\left(0, \gamma_{{\mathcal A}}^2\right),\\
	\bm{\mathcal{N}}\left(0,(1-\alpha)\gamma_{\hat{\mathcal A}}^2\right)\xrightarrow{P}0.
	\end{align*}
	Therefore from \cref{eq:reachedoneobj,eq:fiugwebhdndpwemd} and using \cite[Theorem 3.1]{MR1700749} we obtain that
	\begin{equation*}
		\left(\frac{\bm{S}_{n}}{\sqrt{n}}\bigg|{\bm{X}}_{n+1}=0\right)\xrightarrow{d}\bm{\mathcal{N}}\left(0,\gamma_{{\mathcal A}}^2\right),
	\end{equation*}
	concluding the proof.
\end{proof}

\subsection{Proof of the central limit theorem for permutations }\label{sect:main_thm_proof}

We can now prove our main result.

\begin{proof} [Proof of Theorem \ref{thm:main_thm_CLT}]
	We fix $h\in\Z_{\geq 1}$ and a pattern $\pi\in\mathcal{S}^{h}.$ Let $\bm{\sigma}_n$ be a uniform permutation of size $n$ in $\mathcal C$. From \cref{ass2} (w.l.o.g.\ we can assume $\beta=0$) and \cref{ass3} we have that, conditioning on $\left\{({\bm{X}}_{i})>c(h)\text{ for all } i\in[j+1,j+h]\right\}$,
	\begin{equation*}
	\pat_{[j+1,j+h]}(\bm{\sigma}_n)=\Pat\big(({\bm{Y}}_i^{\bm c})_{i\in[j,j+h-1]}\big|({\bm{X}}_i)_{i\in[2,n]}\geq 0,{\bm{X}}_{n+1}=0\big).
	\end{equation*}
	Therefore, for any sequence of integers $(a_n)_{n\in\Z_{>0}}$, conditioning on 
	$$\left\{(\bm{X}_{i})>c(h) \text{ for all } i\in[a_n,n-a_n]\right\},$$ 
	we can rewrite $\coc(\pi,\bm{\sigma}_n)$ as
	\begin{equation}\label{eq:rewriting_coc}
	\left(\sum_{j\in[1,n-h+1]}\mathds{1}_{\left\{({\bm{Y}}_i^{\bm c})_{i\in[j,j+h-1]}\in\Pat^{-1}(\pi)\right\}}\Bigg|({\bm{X}}_i)_{i\in[2,n]}\geq 0,{\bm{X}}_{n+1}=0\right)+O(a_n).
	\end{equation}
	
	Now, using \cref{lem:lemma3}, we have that
	\begin{equation*}
	\left(\frac{\sum_{j\in[1,n-h+1]}\mathds{1}_{\left\{({\bm{Y}}^{\bm c}_i)_{i\in[j,j+h-1]}\in\Pat^{-1}(\pi)\right\}}-n\cdot \mu_{\pi}}{\sqrt n}\Bigg|({\bm{X}}_i)_{i\in[2,n]}\geq 0,{\bm{X}}_{n+1}=0\right)\stackrel{d}{\longrightarrow}\bm{\mathcal{N}}(0,\gamma_{\pi}^2),
	\end{equation*}
	where the expressions of $\mu_{\pi}$ and $\gamma_{\pi}^2$ given in the statement of \cref{thm:main_thm_CLT} follow from the statement of \cref{lem:lemma3}.
	 Since thanks to Proposition \ref{prop:prop_labels_big}, we can choose $a_n=o(\sqrt n)$ and such that
	\begin{equation*}
	\P\left({\bm{X}}_{i}>c\text{ for all } i\in[a_n,n-a_n]\Big|({\bm{X}}_{j})_{j\in[2,n]}\geq 0,{\bm{X}}_{n+1}=0\right)\to 1,
	\end{equation*}
	then we can conclude that
	\begin{equation*}
	\frac{\coc(\pi,\bm{\sigma}_n)-n\cdot \mu_{\pi}}{\sqrt n}\stackrel{d}{\longrightarrow}\bm{\mathcal{N}}(0,\gamma_{\pi}^2).\qedhere
	\end{equation*}
\end{proof}

\begin{proof} [Proof of \cref{corl:main_thm}]
	From \cref{detstrongbsconditions} the quenched statement is equivalent to the existence of non-negative real numbers $(\Lambda_\pi)_{\pi\in\mathcal{S}}$ such that
	$$\frac{{\coc}(\pi,\bm{\sigma}_n)}{n}\stackrel{P}{\longrightarrow}\Lambda_{\pi},\quad\text{for all}\quad \pi\in\mathcal{S}.$$ 
	This immediately follows from Theorem \ref{thm:main_thm_CLT} with $\Lambda_\pi=\mu_{\pi}$. The annealed statement is a consequence of the quenched one (see \cref{thm:local_conv_perm_charact}).
\end{proof}

\subsection{Construction of the limiting random total order}
\label{sec:construction}

We now exhibit an explicit construction of the limiting object $\bm{\sigma}^{\infty}_{\mathcal{C}}$ (appearing in the statement of \cref{corl:main_thm}) as a random total order $\bm{\preccurlyeq}_{\mathcal{C}}$ of $\mathbb{Z}.$ We recall that a presentation of the local topology for permutations is given in \cref{sect:appe_local_topo}.

Consider a bi-infinite sequence of i.i.d.\ random variables $({\bm{Y}}^{\bm c_i}_i)_{i\in\Z}$ distributed as the random variable ${\bm{Y}}^{\bm c}$ defined above Assumption \ref{ass2}. We set, for all $h\in\Z_{>0},$
$$\bm{\nu}_{2h+1}\coloneqq\Pat\big(({\bm{Y}}^{\bm c_i}_i)_{i\in[-h,h]}\big),$$
where $\Pat$ is defined in \cref{ass3}.
Noting that the sequence of rooted permutations $(\bm{\nu}_{2h+1},h+1)_{h\in\Z_{>0}}$ is a.s.\ consistent and applying \cite[Proposition 2.12]{borga2018local} we have that the sequence $(\bm{\nu}_{2h+1},h+1)_{h\in\Z_{>0}}$ determines a unique random total order $(\Z,\bm{\preccurlyeq}_{\mathcal{C}})$ such that, for all $h\in\Z_{>0},$ $$r_h(\Z,\bm{\preccurlyeq}_{\mathcal{C}})=(\bm{\nu}_{2h+1},h+1),\quad\text{a.s.}$$
where $r_h$ is the restriction function around the root defined in \cref{rhfunct}.                                                                                                                                                                                                                                                                                                                                                                                                                                                                                                                                                                                                                                                                                                                                                                                                                                                                                                                                                                                                                                                                                                                                                                                                                                                                                                           
\begin{prop}
	\label{explicit_constuction_limit}
	Let $(\Z,\bm{\preccurlyeq}_{\mathcal{C}})$ be the random total order defined above and $\bm{\sigma}^\infty_{\mathcal{C}}$ be the limiting rooted permutation defined in \cref{corl:main_thm}. Then
	$$(\Z,\bm{\preccurlyeq}_{\mathcal{C}})\stackrel{d}{=}\bm{\sigma}^\infty_{\mathcal{C}}.$$
\end{prop}

\begin{proof}
	From \cite[Observation 2.23]{borga2018local}, the set of balls 
	\begin{equation*}
	\mathcal{A}=\Big\{B\big((A,\preccurlyeq),2^{-h}\big):h\in\Z_{>0},(A,\preccurlyeq)\in\Sri\Big\}
	\end{equation*}
	is a separating class for the space of rooted permutations $(\Sri,d),$ i.e.\ if
	two measures agree on $\mathcal{A}$ then they are the same. Given a ball $B\left((A,\preccurlyeq),2^{-h}\right)\in\mathcal{A}$ and a rooted permutation $(A',\preccurlyeq')\in\Sri$ then 
	\begin{equation}\label{eq:wigfiuwgfwe}
	\P\left((A',\preccurlyeq')\in B\left((A,\preccurlyeq),2^{-h}\right)\right)=\P\left(r_h(A',\preccurlyeq')= r_h(A,\preccurlyeq)\right)=\P\left(r_h(A',\preccurlyeq')=(\pi,h+1)\right),
	\end{equation}
	where $\pi\in\mathcal{S}^{2h+1}$ is the unique permutation such that $r_h(A,\preccurlyeq)=(\pi,h+1)$ (it exists thanks to the discussion above \cref{defn:inf_rooted}).

	 Note that for all $h\in\Z_{>0},$ for all $\pi\in\mathcal{S}^{2h+1},$
	\begin{equation*}
	\P\big(r_h(\Z,\bm{\preccurlyeq}_{\mathcal{C}})=(\pi,h+1)\big)=\P\big(\Pat({\bm{Y}}^{\bm c}_{-h},\dots,{\bm{Y}}^{\bm c}_{h})=\pi\big),
	\end{equation*}	 
	and from \cref{corl:main_thm},
	\begin{equation*}
	\P\big(r_h(\bm{\sigma}^\infty_\mathcal{C})=(\pi,h+1)\big)=\P\big(\Pat({\bm{Y}}^{\bm c}_1,\dots,{\bm{Y}}^{\bm c}_{2h+1})=\pi\big).
	\end{equation*}	 
	 Since the random variables $({\bm{Y}}^{\bm c}_i)_{i\in\Z}$ are i.i.d., the two equations above together with \cref{eq:wigfiuwgfwe} imply that $(\Z,\bm{\preccurlyeq}_{\mathcal{C}})$ and $\bm{\sigma}^\infty_{\mathcal{C}}$ agree on $\mathcal{A}$ and so they have the same distribution.
\end{proof}

\section{Some families of permutations that satisfy the three assumptions}\label{sect:examples_gen_tree_Ok}

In this section we show that several known families of permutations encoded by generating trees verify the three Assumptions \ref{ass1}, \ref{ass2} and \ref{ass3} of \cref{thm:main_thm_CLT} and \cref{corl:main_thm}. Moreover, for the family of $\{1423,4123\}$-avoiding permutations we exhibit an explicit description of the function $\Pat$ appearing in Assumption \ref{ass3}.

\subsection{\{1423,4123\}-avoiding permutations} 
\label{sec:1423,4123}

We already saw in Examples \ref{exmp:1423_4123_av_part0}, \ref{exmp:1423_4123_av} and \ref{exmp:1423_4123_av_part2}
several properties of \{1423,4123\}-avoiding permutations and their generating trees.
In particular, we saw that all the labels of the generating tree are greater or equal than $3$ and the label of the root is 2. Therefore \cref{ass1} holds. Solving  (with standard computations) the system in \cref{eq:frhweibdowenpen} and using \cref{prop:ehibvreibcpiwrnfc}, we obtain the following.

\begin{prop}\label{prop:sampling_permu}
	 Assumption \ref{ass2} holds with $\alpha_y=(2-\sqrt{2})^{-y}(3-2\sqrt{2})$ for all $y\in\Z_{\leq0}$, $\alpha_{1}=2(2-\sqrt 2)^{-1}(3-2\sqrt 2)=(2-\sqrt{2})$ and, $c_y=1$ for all $y\in\Z_{\leq0}$ and $c_1=2$.
\end{prop}

We now show that also \cref{ass3} is satisfied by this family.
\begin{prop} 
	\label{prop:comb_lemma}	
	Assumption \ref{ass3} is satisfied with $c(h)=h+1.$
\end{prop}

Note that Propositions \ref{prop:sampling_permu} and \ref{prop:comb_lemma} imply that \cref{thm:main_thm_CLT} and \cref{corl:main_thm} hold for \{1423,4123\}-avoiding permutations.

\medskip

In order to prove \cref{prop:comb_lemma} we have to introduce some more notation.
\begin{obs}
	\label{obs:corresp_jump_val}
	We make clearer the role of the jumps in the generating tree for $\{1423,4123\}$-avoiding permutations. Suppose that $\sigma$ has $k$ active sites, then the possible jumps from $\sigma$ to a child of $\sigma$ are $\textcolor{blue}{[+1^B},-(k-3),-(k-4),\dots,-1,0,\textcolor{orange}{+1^T}].$ The child corresponding to the the $i$-th jump is then obtained by appending a final value  in the $i$-th active site of $\sigma$ (ordered from bottom to top). The situation is summarized in \cref{corresp_jump_pos}.
	
	\begin{figure}[htbp]
	\begin{center}
		\includegraphics[scale=0.79]{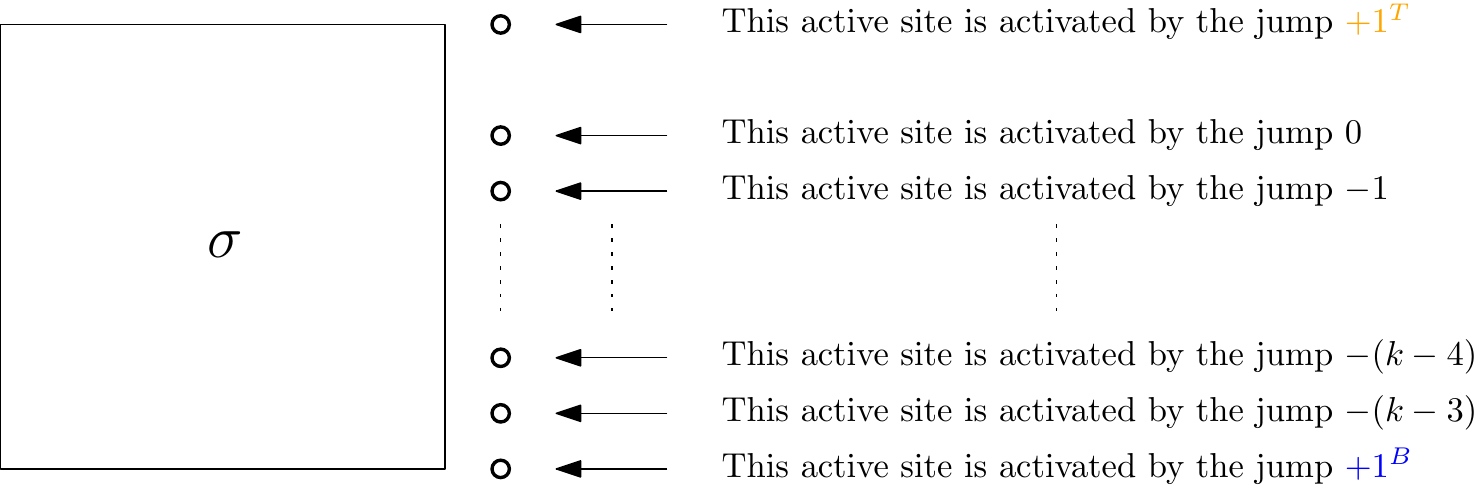}\\
		\caption{The correspondence between (colored) jumps and active sites of a permutation $\sigma\in\Av(1423,4123)$ with $k$ active sites. \label{corresp_jump_pos}}
	\end{center}
	\end{figure}
\end{obs}

We now introduce a construction useful for the proof of Proposition \ref{prop:comb_lemma}. It might be useful to compare what follows with the consecutive Example \ref{exmp:list_act_ssites}. Given a permutation $\sigma\in\Av^n(1423,4123)$ we will define, for all $h\geq 1,$ a list $S_h(\sigma)$ that records the position of the active sites in $\sigma$ with respect to the position of the last $h$ appended final values. Moreover, this list $S_h(\sigma)$ will record the correspondence between (colored) jumps in the generating tree and the positions of the new final values.

We call $d_1,d_2,\dots,d_h$ the $h$ right-most dots in the diagram of $\sigma$ from bottom to top. 
We also assume that the active sites of $\sigma$ are positioned from top to bottom as follows:
\begin{itemize}
	\item There is one active site at the top of $\sigma$ (this is always the case for $\{1423,412\}$-avoiding permutations) corresponding to the jump label $y^{c}=\textcolor{orange}{+1^T}$.
	\item There are $q_h\geq0$ active sites between the top active site and $d_h$. These active sites corresponds to jump labels $y\in[-q_h+1,0].$
	\item For all $ h-1\geq j \geq 1,$ there are $q_j\geq0$ active sites between $d_{j+1}$ and $d_j$. These active sites corresponds to jump labels $y\in[-(\sum_{i=j}^hq_i)+1,-(\sum_{i=j+1}^hq_i)]$.
	\item There are $q_0\geq0$ active sites between $d_1$ and the bottom active site. These active sites corresponds to jump labels $y\in[-(\sum_{i=0}^hq_i)+1,-(\sum_{i=1}^hq_i)]$.
	\item There is an active site at the bottom of $\sigma$ (this is always the case for $\{1423,412\}$-avoiding permutations) corresponding to the jump label $y^{c}=\textcolor{blue}{+1^B}.$	
\end{itemize}

Denoting with $s_j\coloneqq-(\sum_{i=j}^hq_i)$, and assuming that the last $h$ dots $d_1,d_2,\dots,d_h$ form the pattern $\pi$, i.e.\ $\pat_{[n-(h-1),n]}(\sigma)=\pi,$  we set $S_h(\sigma)$ to be the list
\begin{equation*}
\left[[s_0+1,s_1],\mathcircled{\pi^{-1}_1},\dots,\mathcircled{\pi^{-1}_j},[s_j+1,s_{j+1}], \mathcircled{\pi^{-1}_{j+1}},\dots,\mathcircled{\pi^{-1}_{h-1}} ,[s_{h-1}+1,s_h],\mathcircled{\pi^{-1}_h}, [-s_h+1,0]\right],
\end{equation*}
where $\mathcircled{\pi^{-1}_{i}}$ simply denotes the column index of the dot $d_i$ (when the last $h$ columns are numbered from 1 to $h$). We also remark that we are using the following convention: every interval of the form $[a,b]$, with $a>b,$ has to be understood as the empty interval.

The list $S_h(\sigma)$ is therefore composed by
\begin{itemize}
	\item(possibly empty) intervals recording the correspondence between active sites and jumps, whose elements will be called \emph{positions};
	\item circled positive integers recording the relative position of the last $h$ values of $\sigma$ which will be called \emph{indices}.
\end{itemize}
Moreover, the alternation of intervals and circled positive integers records the relative positions between active sites of $\sigma$ and the last $h$ values of $\sigma.$ Finally, note also that the list $S_h(\sigma)$ does not record the position of the bottom and top active sites since for each permutation $\sigma$ they would always be at the beginning and at the end of the list.
 
\begin{exmp}
	\label{exmp:list_act_ssites}
	We consider a permutation $\sigma$ such that the position of the last $h=5$ values induce the pattern $\pi=35214$ ($\pi^{-1}=43152$) with $k=8$ active sites distributed as follows: 
	\begin{equation*}
	\sigma=
	\begin{array}{lcr}
	\includegraphics[scale=0.8]{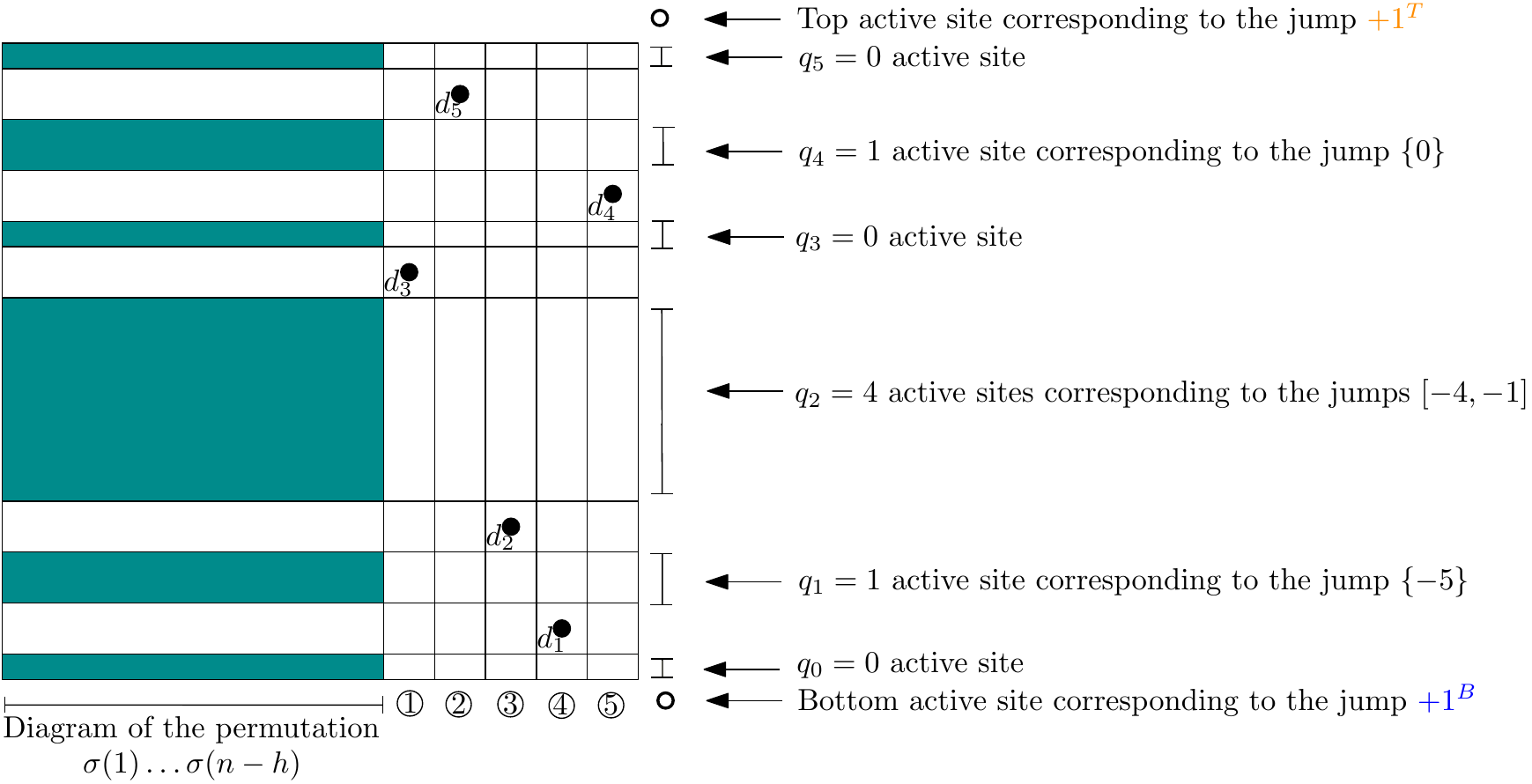}\\
	\end{array}
	\end{equation*}
	Then reading from bottom to top the positions of the active sites and the column index of the point $d_i,$ we have
	\begin{align*}
	S_5(\sigma)&=\big[\emptyset,\mathcircled{4},\{-5\},\mathcircled{3},[-4,-1],\mathcircled{1},\emptyset,\mathcircled{5},\{0\},\mathcircled{2},\emptyset\big]\\
	&=\big[\mathcircled{4},\{-5\},\mathcircled{3},[-4,-1],\mathcircled{1},\mathcircled{5},\{0\},\mathcircled{2}\big].
	\end{align*}
\end{exmp}

\bigskip

We finally introduce a small modification of the list $S_h(\sigma).$ The benefit of this modification will be clearer in Example \ref{exemp:messup} (see in particular the discussion after \cref{eq:key_exemp}). Given an integer $\ell\geq 1$ we denote by $S^\ell_h(\sigma)$ the list obtained from $S_h(\sigma)$ forgetting all the $\ell-1$ left-most positions (corresponding to the $\ell-1$ most negative jumps) and replacing the interval involving the $\ell$-th left-most position, say $[-p,-q],$ with the interval $(-\infty,-q].$ 

\begin{exmp}\label{exemp:dhjswvfciewbfceqlo}
	Using the list $S_5(\sigma)=\big[\mathcircled{4},\{-5\},\mathcircled{3},[-4,-1],\mathcircled{1},\mathcircled{5},\{0\},\mathcircled{2}\big]$ from Example \ref{exmp:list_act_ssites} we have that
	\begin{equation*}
	\begin{split}
	S^1_5(\sigma)=&\big[\mathcircled{4},(-\infty,-5],\mathcircled{3},[-4,-1],\mathcircled{1},\mathcircled{5},\{0\},\mathcircled{2}\big],\\
	S^2_5(\sigma)=&S^3_5(\sigma)=S^4_5(\sigma)=S^5_5(\sigma)=\big[\mathcircled{4},\mathcircled{3},(-\infty,-1],\mathcircled{1},\mathcircled{5},\{0\},\mathcircled{2}\big],\\
	S^6_5(\sigma)=&\big[\mathcircled{4},\mathcircled{3},\mathcircled{1},\mathcircled{5},(-\infty,0],\mathcircled{2}\big].\\
	\end{split}
	\end{equation*}
\end{exmp}

\begin{obs}
	\label{obs:key_for_proof}
	For every $\ell,h\geq 1$, $\sigma\in\Av^n(1423,4123),$ the list $S^\ell_h(\sigma)$ determines the final consecutive pattern $\pat_{[n-(h-1),n]}(\sigma).$ Indeed, it is equal to the inverse of the permutation obtained reading the indices in $S^\ell_h(\sigma)$ from left to right. We denote this pattern with $\Pat\left(S^\ell_h(\sigma)\right).$
\end{obs}
\begin{exmp}
		Using for instance the list $S^2_5(\sigma)=\big[\mathcircled{4},\mathcircled{3},[-\infty,-1],\mathcircled{1},\mathcircled{5},\{0\},\mathcircled{2}\big]$ from Example \ref{exemp:dhjswvfciewbfceqlo} then 
		$$\pat_{[|\sigma|-4,|\sigma|]}(\sigma)=(43152)^{-1}=35214.$$
\end{exmp}

We now investigate in a specific example the result stated in Proposition \ref{prop:comb_lemma}. Specifically, we show how the function $\Pat$ acts on our running example of $\Av(1423,4123)$. The example is rather long but it covers most of the cases of the proof of the consecutive Lemma \ref{lem:ifwirwbfiurw} that is quite technical.
\begin{exmp}
	\label{exemp:messup}
	Let $n\geq6$ and $(k_i^{c})_{i\in[n]}$ be a sequence of colored labels with $(y_i^{c})_{i\in[n-1]}$ the corresponding sequence of colored jumps. Let $h=6$ and $1\leq m\leq n-5.$ Assume that $k_i>h+1=7$ for all $i\in[m,m+5]$ and that
	$$(y_i^c)_{i\in{[m-1,m+4]}}=(-2,\textcolor{blue}{+1^B},\textcolor{blue}{+1^B},\textcolor{orange}{+1^T},\textcolor{orange}{+1^T},-7).$$
	We denote by $(p_i)_{i\in[6]}$ the six dots from left to right in the diagram of $\GG((k_i^c)_{i\in[n]})$ corresponding to the indices $[m,m+5]$ (note that the dots $d_i$ in Example \ref{exmp:list_act_ssites} were ordered from bottom to top and so we use a different notation in this example).
	
	We now reconstruct the consecutive pattern $\pat_{[m,m+5]}\big(\GG((k_i^c)_{i\in[n]})\big)$ reading the jumps in $(y_i^c)_{i\in{[m-1,m+4]}}$ from left to right as follows:
	\begin{itemize}
		\item The first jump is $y_{m-1}^c=-2$, therefore $\GG((k_i^c)_{i\in[m]})$ has one active site above $p_1$ and all the other active sites below $p_1$ (see Example \ref{exmp:1423_4123_av} for a reminder on the behavior of the active sites when a new dot is appended). We note that $p_2$ will be below $p_1$ if and only the next jump $y^c$ satisfies $y\leq 0$ or $y^c=\textcolor{blue}{+1^B},$ otherwise (if $y^c=\textcolor{orange}{+1^T}$) $p_2$ will be above $p_1$. Therefore
		\begin{equation*}
		\GG((k_i^c)_{i\in[m]})=
		\begin{array}{lcr}
		\includegraphics[scale=0.79]{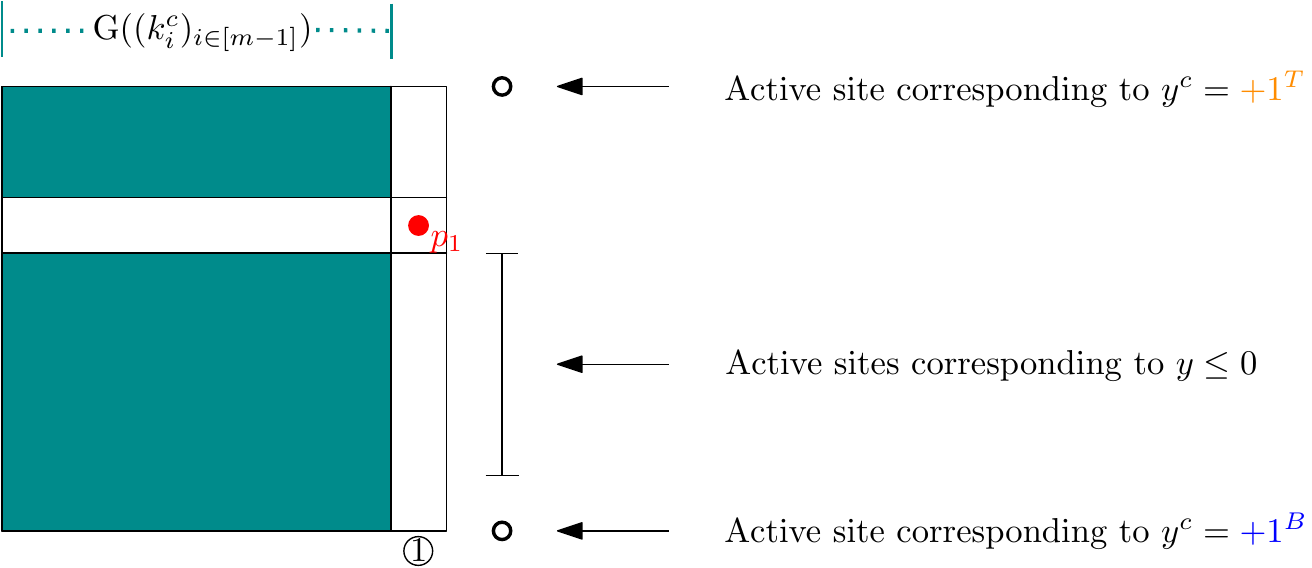}$$
		\end{array}
		\end{equation*}
		The situation is recorded by the list
		$$S_1^1\big(\GG((k_i^c)_{i\in[m]})\big)=\big[(-\infty,0],\mathcircled{1}\big].$$
		
		\item The second jump is $y_{m}^c=\textcolor{blue}{+1^B}$ and so we append to $\GG((k_i^c)_{i\in[m]})$ the minimal value. Consequently, $\GG((k_i^c)_{i\in[m+1]})$ has an active site at the top (corresponding to $y^c=\textcolor{orange}{+1^T}$), an active site at the bottom (corresponding to $y^c=\textcolor{blue}{+1^B}$) and all the other active sites between $p_2$ and $p_1$ (corresponding to $y\leq 0$).
		Therefore
		\begin{equation*}
		\quad\GG((k_i^c)_{i\in[m+1]})=
		\begin{array}{lcr}
		\includegraphics[scale=0.78]{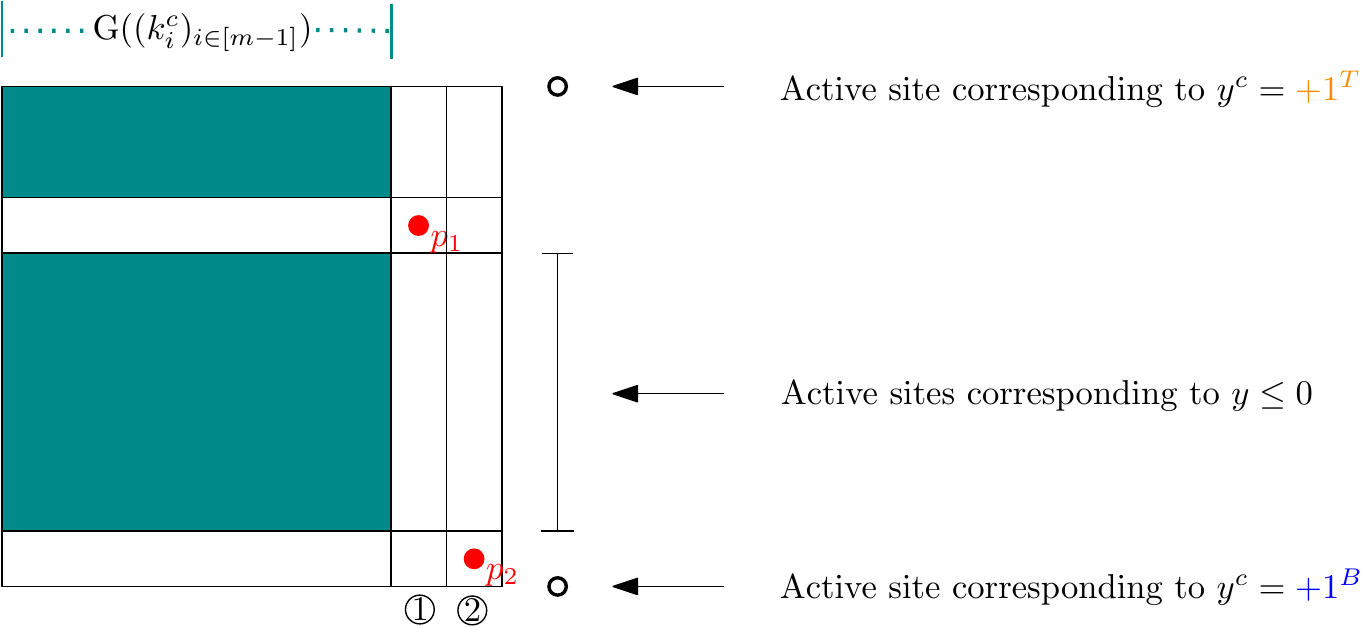}$$
		\end{array}
		\end{equation*}
		The situation is recorded by the list
		$$S_2^2\big(\GG((k_i^c)_{i\in[m+1]})\big)=\big[\mathcircled{2},(-\infty,0],\mathcircled{1}\big].$$
		
		\item The third jump is $y_{m+1}^c=\textcolor{blue}{+1^B}$ and so we append to $\GG((k_i^c)_{i\in[m+1]})$ the minimal value. As a consequence, $\GG((k_i^c)_{i\in[m+2]})$ has
		\begin{itemize}
			\item an active site at the top (corresponding to $y^c=\textcolor{orange}{+1^T}$);
			\item an active site at the bottom (corresponding to $y^c=\textcolor{blue}{+1^B}$);
			\item an active site between $p_3$ and $p_2$. Note that this active site will be not activated by the next $3$ jumps. Indeed, this active site (corresponding to the label $3$) could be activated by the next 3 jumps only if at least one of the next 3 labels has value smaller than $3+2=5.$ This cannot happen since, by assumption, $k_i>7$ for all $i\in[m,m+5]$;
			\item all the other active sites between $p_2$ and $p_1$ (corresponding to $y\leq 0$).
		\end{itemize}
		Therefore
		\begin{equation*}
		\quad\quad\quad\GG((k_i^c)_{i\in[m+2]})=
		\begin{array}{lcr}
		\includegraphics[scale=0.65]{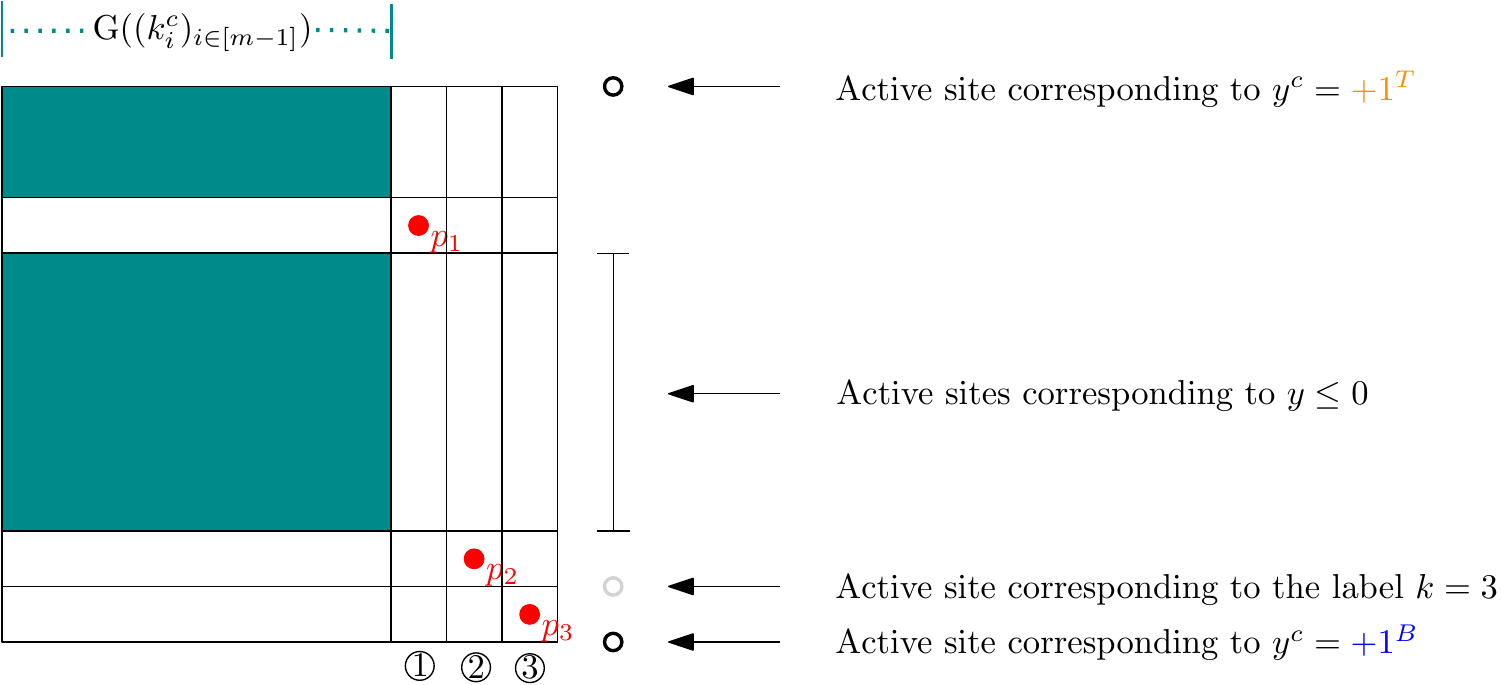}$$
		\end{array}
		\end{equation*}
		The situation is recorded by the list
		\begin{equation}
			\label{eq:key_exemp}
			S_3^3\big(\GG((k_i^c)_{i\in[m+2]})\big)=\big[\mathcircled{3},\mathcircled{2},(-\infty,0],\mathcircled{1}\big].
		\end{equation}
		We highlight that in this situation we would not be able to determine the complete list $S_3\big(\GG((k_i^c)_{i\in[m+2]})\big)$ only knowing the sequence of jumps $(y_i^c)_{i\in{[m-1,m+4]}}$. Indeed, in order to know which jump would activate the active site between $p_3$ and $p_2$  we would need to know the actual value of the label $k_{m+2}.$ However, the truncation $S_3^3\big(\GG((k_i^c)_{i\in[m+2]})\big)$ is determined only by the jumps as explained above.
		
		\item The fourth jump is $y_{m+2}^c=\textcolor{orange}{+1^T}$ and so we append to $\GG((k_i^c)_{i\in[m+2]})$ the value $m+3$. As a consequence, $\GG((k_i^c)_{i\in[m+3]})$ has
		\begin{itemize}
			\item an active site at the top (corresponding to $y^c=\textcolor{orange}{+1^T}$);
			\item an active site at the bottom (corresponding to $y^c=\textcolor{blue}{+1^B}$);
			\item an active site between $p_3$ and $p_2$ (that we can forget thanks to the previous explanation);
			\item an active site between $p_1$ and $p_4$ (corresponding to $y=0$);
			\item all the other active sites between $p_2$ and $p_1$ (corresponding to $y\leq -1$).
		\end{itemize}
		Therefore
		\begin{equation*}
		\quad\quad\quad\;\GG((k_i^c)_{i\in[m+3]})=
		\begin{array}{lcr}
		\includegraphics[scale=0.65]{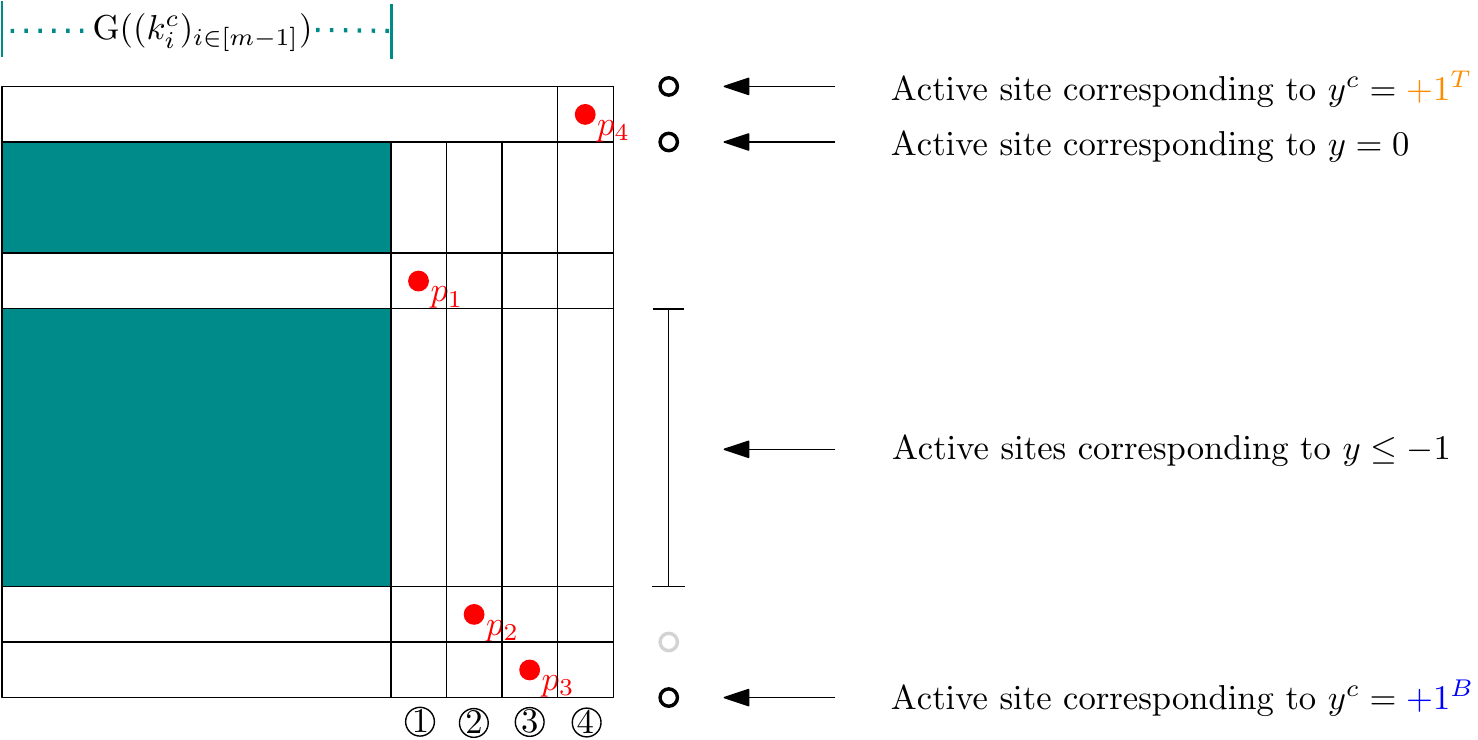}$$
		\end{array}
		\end{equation*}
		The situation is recorded by the list
		$$S_4^4\big(\GG((k_i^c)_{i\in[m+3]})\big)=\big[\mathcircled{3},\mathcircled{2},(-\infty,-1],\mathcircled{1},\{0\},\mathcircled{4}\big].$$
		
		\item The fifth jump is $y_{m+3}^c=\textcolor{orange}{+1^T}$ and so we append to $\GG((k_i^c)_{i\in[m+3]})$ the value $m+4$. As a consequence, $\GG((k_i^c)_{i\in[m+4]})$ has
		\begin{itemize}
			\item an active site at the top (corresponding to $y^c=\textcolor{orange}{+1^T}$);
			\item an active site at the bottom (corresponding to $y^c=\textcolor{blue}{+1^B}$); 
			\item an active site between $p_3$ and $p_2$ (that we can forget thanks to the previous explanation);
			\item an active site between $p_4$ and $p_5$ (corresponding to $y=0$);
			\item an active site between $p_1$ and $p_4$ (corresponding to $y=-1$);
			\item all the other active sites between $p_2$ and $p_1$ (corresponding to $y\leq -2$).
		\end{itemize}
		Therefore
		\begin{equation*}
		\quad\quad\quad\quad\GG((k_i^c)_{i\in[m+4]})=
		\begin{array}{lcr}
		\includegraphics[scale=0.63]{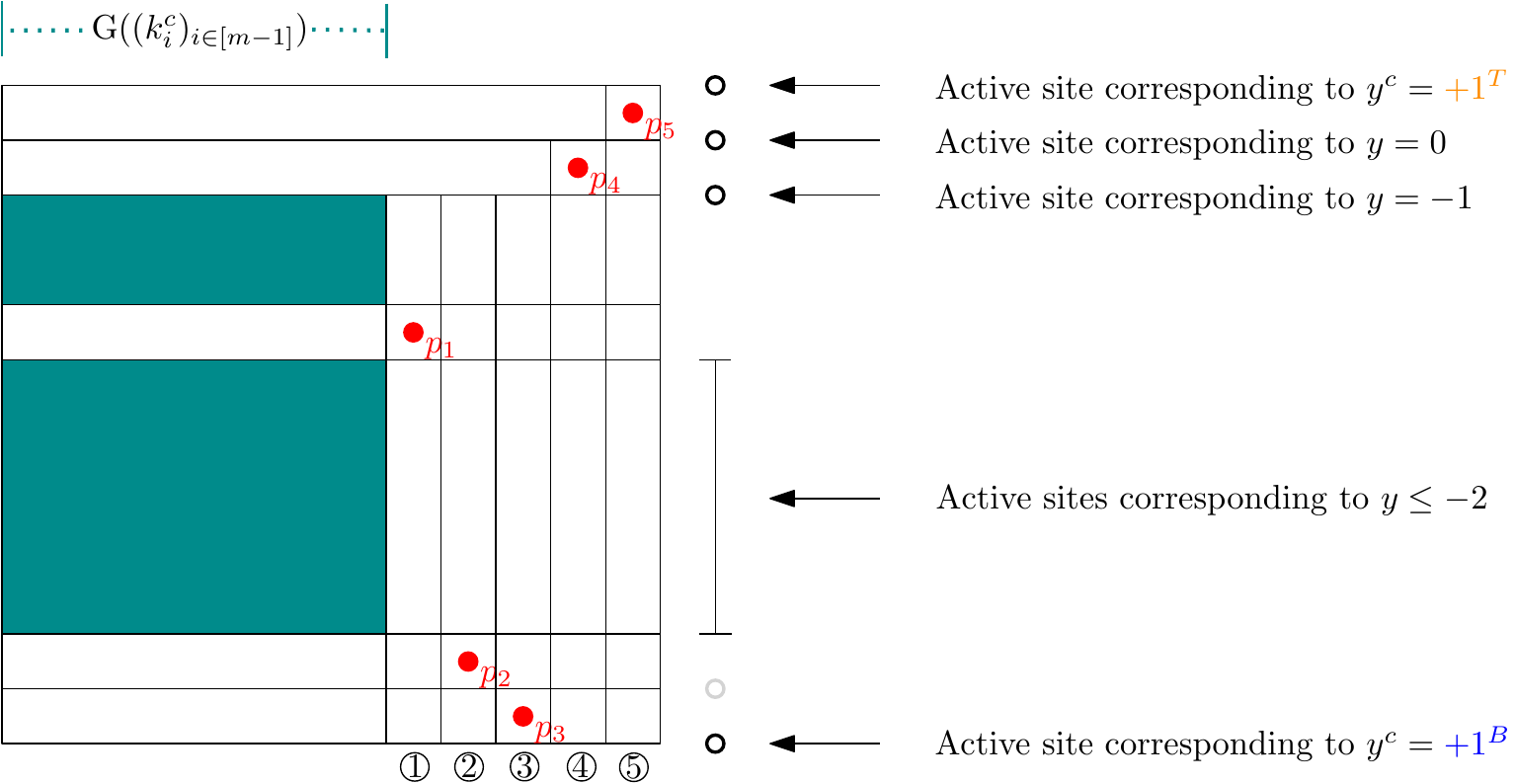}$$
		\end{array}
		\end{equation*}
		The situation is recorded by the list
		$$S_5^5\big(\GG((k_i^c)_{i\in[m+4]})\big)=\big[\mathcircled{3},\mathcircled{2},(-\infty,-2],\mathcircled{1},\{-1\},\mathcircled{4},\{0\},\mathcircled{5}\big].$$
		
		\item The sixth jump is $y_{m+4}^c=-7$ and so we append to $\GG((k_i^c)_{i\in[m+3]})$ a new final value between $p_1$ and $p_2$. As a consequence, $\GG((k_i^c)_{i\in[m+5]})$ has
		\begin{itemize}
			\item an active site at the top (corresponding to $y^c=\textcolor{orange}{+1^T}$);
			\item an active site at the bottom (corresponding to $y^c=\textcolor{blue}{+1^B}$); 
			\item an active site between $p_3$ and $p_2$ (that we can forget thanks to the previous explanation);
			\item all the other active sites between $p_2$ and $p_6$ (corresponding to $y\leq 0$).
		\end{itemize}
		Therefore
		\begin{equation*}
		\quad\quad\quad\;\GG((k_i^c)_{i\in[m+5]})=
		\begin{array}{lcr}
		\includegraphics[scale=0.6]{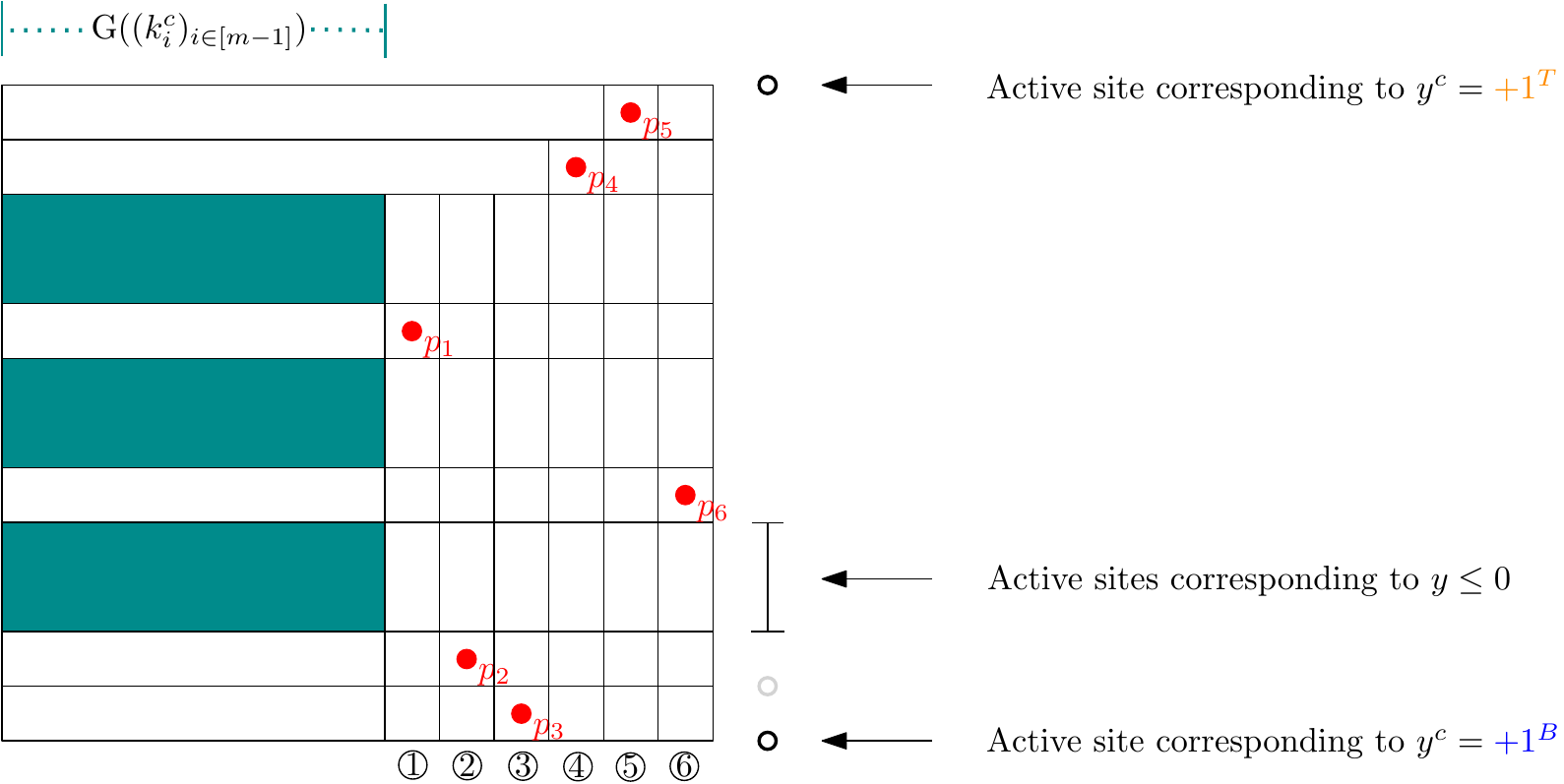}$$
		\end{array}
		\end{equation*}
		The situation is recorded by the list
		$$S_6^6\big(\GG((k_i^c)_{i\in[m+5]})\big)=\big[\mathcircled{3},\mathcircled{2},(-\infty,0],\mathcircled{6},\emptyset,\mathcircled{1},\emptyset,\mathcircled{4},\emptyset,\mathcircled{5}\big].$$
	\end{itemize}
We can conclude that the consecutive pattern $\pat_{[m,m+5]}\big(\GG((k_i^c)_{i\in[n]})\big)$ is determined by the sequence of jumps
$$(y_i^c)_{i\in{[m-1,m+4]}}=(-2,\textcolor{blue}{+1^B},\textcolor{blue}{+1^B},\textcolor{orange}{+1^T},\textcolor{orange}{+1^T},-7).$$
Specifically, $\Pat(-2,\textcolor{blue}{+1^B},\textcolor{blue}{+1^B},\textcolor{orange}{+1^T},\textcolor{orange}{+1^T},-7)=(326145)^{-1}=421563.$
\end{exmp}

Keeping this example in mind, we now prove the following lemma that easily implies Proposition \ref{prop:comb_lemma}.	

\begin{lem}
	\label{lem:ifwirwbfiurw}
	Fix $h\geq 1.$ For every sequence of labels $(k_i^c)_{i\in[n]}$  that satisfies for some $m\in[n+1-h]$ the condition 
	$$k_i>h+1,\quad\text{for all}\quad i\in[m,m+h-1],$$ 
	it is possible to reconstruct the list $S^{h}_{h}\big(\GG((k_i^c)_{i\in[m+h-1]})\big)$
	only from the sequence of jumps $(y_i^c)_{i\in[m-1,m+h-2]}.$
\end{lem}
\begin{rem}
	We highlight that the following proof also establishes an explicit way (that can be easily implemented) to reconstruct the list $S^{h}_{h}\big(\GG((k_i^c)_{i\in[m+h-1]})\big)$ from the sequence of jumps $(y_i^c)_{i\in[m-1,m+h-2]}.$
\end{rem}

\begin{proof}
	Let $(k_i^c)_{i\in[n]}$ be a sequence such that for some $m\in[n+1-h],$ 
	$$k_i^c>h+1,\quad\text{for all}\quad i\in[m,m+h-1].$$
	We prove the statement by induction. We first show that $S^{1}_{1}\big(\GG((k_i^c)_{i\in[m]})\big)$ is determined by the jump label $y_{m-1}^c$ and then we show how to reconstruct the list $S^{j+1}_{j+1}\big(\GG((k_i^c)_{i\in[m+j]})\big)$ from the list $S^{j}_{j}\big(\GG((k_i^c)_{i\in[m+j-1]})\big)$ according to the value of the jump label $y_{m+j-1}^c.$
	
	We explain how to determine $S^{1}_{1}\big(\GG((k_i^c)_{i\in[m]})\big)$ from the jump label $y_{m-1}^c$. The diagram of the permutation $\GG((k_i^c)_{i\in[m]})$ is obtained from the diagram of the permutation $\GG((k_i^c)_{i\in[m-1]})$ appending a new final dot, denoted by $p_1$. We have to distinguish three different cases:
	\begin{itemize}
		\item \underline{$y_{m-1}^c=\textcolor{orange}{+1^T}:$} In this case $p_1$ is appended in the top active site. Therefore, the permutation $\GG((k_i^c)_{i\in[m]})$ has an active site at the top and all the other active sites are below $p_1.$ We can conclude that 
		$$S^{1}_{1}\big(\GG((k_i^c)_{i\in[m]})\big)=\big[(-\infty,0],\mathcircled{1}\big].$$
		\item \underline{$y_{m-1}^c=\textcolor{blue}{+1^B}:$} In this case $p_1$ is appended in in the bottom active site. Therefore, the permutation $\GG((k_i^c)_{i\in[m]})$ has an active site at the bottom and all the other active site are above $p_1.$ We can conclude that 
		$$S^{1}_{1}\big(\GG((k_i^c)_{i\in[m]})\big)=\big[\mathcircled{1},(-\infty,0]\big].$$
		\item \underline{$y_{m-1}\leq 0:$} In this case $p_1$ is appended in one of the active sites different from the top and the bottom one. Therefore, $\GG((k_i^c)_{i\in[m]})$ has an active site at the top and all the other active sites below $p_1$ (indeed, we recall that former active sites above $p_1$ are no longer active in $\GG((k_i^c)_{i\in[m]})$ except the top one, see Example \ref{exmp:1423_4123_av}). We can conclude that 
		$$S^{1}_{1}\big(\GG((k_i^c)_{i\in[m]})\big)=\big[(-\infty,0],\mathcircled{1}\big].$$
	\end{itemize}
	We now show how to reconstruct the list $S^{j+1}_{j+1}\big(\GG((k_i^c)_{i\in[m+j]})\big)$ from the list $S^{j}_{j}\big(\GG((k_i^c)_{i\in[m+j-1]})\big)$ according to the value of the jump label $y_{m+j-1}^c.$ Without loss of generality, we assume that the list $S^{j}_{j}\big(\GG((k_i^c)_{i\in[m+j-1]})\big)$ is equal to
	\begin{equation*}
	\Big[[-a_{j+1},-a_j-1],\mathcircled{b_j},[-a_j,-a_{j-1}-1],\dots,\mathcircled{b_2} ,[-a_2,-a_1-1],\mathcircled{b_1}, [-a_1,0]\Big],
	\end{equation*}
	where $\{b_1,\dots,b_j\}=[1,j]$, $0\leq a_1\leq a_2\leq\dots\leq a_{j+1}\leq +\infty$ and there exists $s\leq j+1$ such that $a_{s}=+\infty$ (with the additional convention on intervals that $[-\infty,-\infty-1]=\emptyset$).
	
	The diagram of the permutation $\GG((k_i^c)_{i\in[m+j]})$ is obtained from the diagram of the permutation $\GG((k_i^c)_{i\in[m+j-1]})$ appending a new final dot, denoted by $p_{j+1}$.
	We distinguish again three different cases:
	\begin{itemize}
		\item \underline{$y_{m+j-1}^c=\textcolor{orange}{+1^T}:$} In this case $p_{j+1}$ is appended in the top active site. Therefore, the active sites  of $\GG((k_i^c)_{i\in[m+j]})$ are the same as $\GG((k_i^c)_{i\in[m+j-1]})$ excluding the one that has been activated. Moreover, there is a new top active site and a new active site immediately below $p_{j+1}.$ We also note that each active site of $\GG((k_i^c)_{i\in[m+j-1]})$ corresponding to a $-r$ jump then corresponds to a $-r-1$ jump. Therefore, the new list $S^{j+1}_{j+1}\big(\GG((k_i^c)_{i\in[m+j]})\big)$ is equal to
		\begin{multline*}
		\Big[[-a_{j+1}\textcolor{red}{-1},-a_j-1\textcolor{red}{-1}],\mathcircled{b_j},[-a_j\textcolor{red}{-1},-a_{j-1}-1\textcolor{red}{-1}],\dots\\
		\dots,\mathcircled{b_2} ,[-a_2\textcolor{red}{-1},-a_1-1\textcolor{red}{-1}],\mathcircled{b_1}, [-a_1\textcolor{red}{-1},0],\textcolor{red}{\mathcircled{j+1}}\Big],
		\end{multline*}
		where we highlighted in red the modifications from the list $S^{j}_{j}\big(\GG((k_i^c)_{i\in[m+j-1]})\big).$
		\item \underline{$y_{m+j-1}^c=\textcolor{blue}{+1^B}:$} In this case $p_{j+1}$ is appended in the bottom active site. Therefore, the active sites  of $\GG((k_i^c)_{i\in[m+j]})$ are the same as $\GG((k_i^c)_{i\in[m+j-1]})$ excluding the one that has been activated. Moreover, there is a new bottom active site and a new active site immediately above $p_{j+1}$ (that by definition of $S^{j+1}_{j+1}$ will be not recorded in the list $S^{j+1}_{j+1}\big(\GG((k_i^c)_{i\in[m+j]})\big)$ since it corresponds to the label 3). We also note that each active site of $\GG((k_i^c)_{i\in[m+j-1]})$ corresponding to a $-r$ jump then still corresponds to a $-r$ jump. Therefore, the new list $S^{j+1}_{j+1}\big(\GG((k_i^c)_{i\in[m+j]})\big)$ is equal to
		\begin{equation*}
		\Big[\textcolor{red}{\mathcircled{j+1}},[-a_{j+1},-a_j-1],\mathcircled{b_j},[-a_j,-a_{j-1}-1],\dots,\mathcircled{b_2} ,[-a_2,-a_1-1],\mathcircled{b_1}, [-a_1,0]\Big].
		\end{equation*}
		\item \underline{$y_{m+j-1}\leq 0:$} In this case $p_{j+1}$ is appended in one of the active sites different from the top and the bottom one. Note that there exists an interval of positions in $S^{j}_{j}\big(\GG((k_i^c)_{i\in[m+j-1]})\big)$ that contains the value $y_{m+j-1},$ say $[-a_t,-a_{t-1}-1]$ for $t\in[1,j+1]$. We distinguish two cases:
		\begin{itemize}
			\item If $a_t\neq-\infty$ then we can immediately conclude that $p_{j+1}$ is positioned between the two dots with column indices $b_t$ and $b_{t-1}.$
			\item If $a_t=-\infty$ there is a potential problem: indeed, $p_{j+1}$ could be placed in one of the forgotten positions in $S^{j}_{j}\big(\GG((k_i^c)_{i\in[m+j-1]})\big)$ (we refer to the discussion before \cref{eq:key_exemp} for an example).
			
			\medskip
			
			\underline{Claim}: The dot $p_{j+1}$ is still positioned between the two dots with column indices $b_t$ and $b_{t-1}.$ 
			\begin{proof}
				The dot $p_{j+1}$ cannot be placed in one of the forgotten positions in the list  $S^{j}_{j}\big(\GG((k_i^c)_{i\in[m+j-1]})\big)$ because this can happen only if the label $k_{m+j}$ is smaller or equal than $3+(j-1)$ (by definition of $S^{j}_{j}$) but by assumption $k_{m+j}>h+1$ and $j\leq h-1.$
			\end{proof}
			
		\end{itemize}

		Therefore the active sites of $\GG((k_i^c)_{i\in[m+j]})$ are:
		\begin{itemize}
			\item the top active site corresponding to the $\textcolor{orange}{+1^T}$ jump;
			\item a new active site immediately below $p_{j+1}$ corresponding to the $0$ jump;
			\item all the active sites of $\GG((k_i^c)_{i\in[m+j-1]})$ that were below the active site replaced by $p_{j+1}.$ Moreover, an active site of $\GG((k_i^c)_{i\in[m+j-1]})$ corresponding to a $-r$ jump that is also active in $\GG((k_i^c)_{i\in[m+j]})$ corresponds to a $-r-y_{m+j-1}$ jump.
		\end{itemize}
		 Therefore, the new list $S^{j+1}_{j+1}\big(\GG((k_i^c)_{i\in[m+j]})\big)$ is equal to
		\begin{multline*}
		\quad\quad\Big[[-a_{j+1}\textcolor{red}{-y_{m+j-1}},-a_j-1\textcolor{red}{-y_{m+j-1}}],\mathcircled{b_j},[-a_j\textcolor{red}{-y_{m+j-1}},-a_{j-1}-1\textcolor{red}{-y_{m+j-1}}],\dots\\
		\dots\mathcircled{b_{t+1}} ,[-a_{t+1}\textcolor{red}{-y_{m+j-1}},\textcolor{red}{0}],\textcolor{red}{\mathcircled{j+1}},\mathcircled{b_t}\dots\mathcircled{b_1}\Big].
		\end{multline*}
	\end{itemize}
This concludes the proof.
\end{proof}

We can finally prove the main result of this section.

\begin{proof}[Proof of Proposition \ref{prop:comb_lemma}]
	Fix $h\geq 1.$ Let $(k_i^c)_{i\in[n]}$ be a sequence such that for some integer $m\in[n+1-h],$  
	$$k_i^c>h+1,\quad\text{for all}\quad i\in[m,m+h-1].$$ 
	Using Observation $\ref{obs:key_for_proof}$ we have that $$\pat_{[m,m+h-1]}\big(\GG((k_i^c)_{i\in[n]})\big)=\Pat\left(S^{h}_{h}\left(\GG((k_i^c)_{i\in[m+h-1]})\right)\right).$$
	Since by Lemma \ref{lem:ifwirwbfiurw} the list $S^{h}_{h}\big(\GG((k_i^c)_{i\in[m+h-1]})\big)$ is only determined by the sequence of jumps $(y_i^c)_{i\in[m-1,m+h-2]}$ we conclude that Assumption \ref{ass3} is satisfied.
\end{proof}

\subsection{Five other families of permutations encoded by the same generating tree as \{1423,4123\}-avoiding permutations} 

In \cite[Theorem 8 and Corollary 9]{kremer2000permutations} it was shown that also the following five families of permutations\footnote{We recall that the various results in this article are stated for families of permutations that grows on the top and not on the right. Nevertheless, it is enough to use the diagonal symmetry of the diagram for pattern-avoiding permutations.} are encoded by the same generating tree used for \{1423,4123\}-avoiding permutations:
\begin{equation}\label{eq:five_more_fam}
\Av(1234,2134),\quad
\Av(1324,3124),\quad
\Av(2314,3214),\quad
\Av(2413,4213),\quad
\Av(3412,4312).
\end{equation}

These families, having the same generating tree as \{1423,4123\}-avoiding permutations, satisfy Assumptions \ref{ass1} and \ref{ass2}.

The main difference between these six families (the five in \cref{eq:five_more_fam} and $\Av(1423,4123)$) is the following. Recall that in Example \ref{exmp:1423_4123_av} we described the behavior of the set $\text{AS}(\sigma)$ for $\sigma\in\Av(1423,4123)$ and how the statistic $|\text{AS}(\cdot)|$ determines the succession rule of the generating tree for $\{1423,4123\}$-avoiding permutations. These properties are specific for each of the six mentioned family and are explicitly described in the proof of \cite[Theorem 8]{kremer2000permutations}. Nevertheless, the specific properties of each of the five above mentioned families still imply Assumption \ref{ass3} (the proof of this fact is similar to the one given for Proposition \ref{prop:comb_lemma} and we omit it). Therefore the main results of our paper also hold for the five families in \cref{eq:five_more_fam}.

\begin{rem}
	In the paper of Kremer \cite{kremer2000permutations}, four further families of permutations enumerated using the same generating tree as the one for \{1423,4123\}-avoiding permutations are mentioned. For instance the family of separable permutations, i.e.\ permutations avoiding the patterns $2413$ and $3142$ (this family was first studied by West in \cite[Lemma 4.1.]{MR1360119}). We highlight that \cref{ass3} is not satisfied by these four families of permutations. Indeed, given a permutation $\sigma$ in one of these four families, the behavior of consecutive patterns in $\sigma$ cannot be determined from the behavior of consecutive increments in the corresponding walk $\GG(\sigma)$. In other words, local properties of $\sigma$ cannot be deduced form local properties of $\GG(\sigma)$.
\end{rem}

\subsection{Permutations avoiding a pattern of size three}\label{sect:three_patterns}

From \cite[Lemma 2.2]{MR1360119} we have that the family $\Av(123)$ is encoded by a generating tree determined by the following succession rule
\begin{equation}\label{eq:gen_rule_fam_A2}
\begin{cases} 
\text{Root label}: (2) \\
(k)\to (2),(3),\dots,(k+1). 
\end{cases}
\end{equation}

Note that the labels of the generating tree are greater or equal than $2$ but the label of the root is not equal to $1$. In order to fix this small issue, it is enough to prepend a new root labeled by $1$ to the original root of the generating tree. Then \cref{ass1} holds.

In this way, every permutation of size $n$ is bijectively encoded by a sequence of $n+1$ labels $(1,k_1=2,k_2,\dots,k_n)$ such that the sequence $(k_1=2,k_2,\dots,k_n)$ is consistent with the succession rule in \cref{eq:gen_rule_fam_A2}. 
We still denote by $\GG$ this bijection between sequences of labels and permutations in the generating tree.

Solving the system in \cref{eq:frhweibdowenpen} and using the same proof of \cref{prop:ehibvreibcpiwrnfc} we obtain the following.

\begin{prop}\label{prop:sampling_permu2}
	Assumption \ref{ass2} holds with $\alpha_y=2^{-2+y}$ for all $y\in\Z_{\leq 1}$, and $c_y=1$ for all $y\in\Z_{\leq1}$. Indeed, if $\bm\sigma_n$ is a uniform 123-avoiding permutation of size $n$, then
	\begin{equation}\label{eq:non-colored-walk}
	\GG^{-1}(\bm\sigma_n)\stackrel{d}{=}({\bm{X}}_{i}|({\bm{X}}_{j})_{j\in[2,n+1]}\geq 2,{\bm{X}}_{n+2}=2)_{i\in[n+1]},
	\end{equation}
where ${\bm{X}}$ is a random walk with jump distribution $(\alpha_y)_{y\in\Z_{\leq 1}}$ starting at $1$ at time 1.
\end{prop}

\begin{rem}
 Note that the random walk ${\bm{X}}$ considered in \cref{eq:non-colored-walk} is \emph{not} colored. Indeed, there is no need of considering colored walks as there are no repeated labels in the succession rule in \cref{eq:gen_rule_fam_A2}.
\end{rem}

Assumption \ref{ass3} follows with a proof similar to the one given for Proposition \ref{prop:comb_lemma} using the specific behavior of the set $\text{AS}(\sigma)$ for $\sigma\in\Av(123)$ and how the statistic $|\text{AS}(\cdot)|$ determines the succession rule of the generating tree (for that we refer to the proof of \cite[Lemma 2.2]{MR1360119}). 

Therefore the main results of our paper hold for the family $\Av(123)$.

\medskip

We now consider the family $\Av(132)$. Since in \cite[Lemma 2.4]{MR1360119} it is shown that the family $\Av(132)$ has the same generating tree as the one for 123-avoiding permutations, then Assumptions \ref{ass1} and \ref{ass2} are satisfied. Once again Assumption \ref{ass3} follows using the specific behavior of the set $\text{AS}(\sigma)$ for $\sigma\in\Av(132)$ (for that we refer to the proof of \cite[Lemma 2.4]{MR1360119}).

\subsection{Two families of permutations avoiding generalized patterns}\label{sect:strange_patterns}

In \cite{MR2376115}, Elizalde constructed generating trees with one, two, and three-dimensional labels for some classes of permutations avoiding generalized patterns of size three and four. In particular, the  following two families of permutations satisfy all our assumptions:
\begin{itemize}
	\item $\mathcal{A}=\{213,\bar{2}\underbracket[.5pt][1pt]{31}\}$-avoiding permutations;
	\item $\mathcal{B}=\{213,\bar{2}^{o}\underbracket[.5pt][1pt]{31}\}$-avoiding permutations.
\end{itemize}
Both families are subsets of $\Av(213)$. We now explain the meaning of the second (generalized) pattern in each family:
\begin{itemize}
	\item For the first family, we additionally require that every permutation $\sigma\in\mathcal{A}$ avoids the generalized pattern $\bar{2}\underbracket[.5pt][1pt]{31}$, i.e.\ every descent in $\sigma$ (an occurrence of the generalized pattern $\underbracket[.5pt][1pt]{21}$) is part of an occurrence of $2\underbracket[.5pt][1pt]{31}$; equivalently, for any index $i$ such that $\sigma_i>\sigma_{i+1}$ there is an index $j<i$ such that $\sigma_i>\sigma_j>\sigma_{i+1}$.The bar indicates that the 2 is forced whenever a 31 occurs. For example, the permutation $4627513$ avoids $\bar{2}\underbracket[.5pt][1pt]{31}$, but $2475613$ does not.
	\item For the second family, we additionally require that every permutation $\sigma\in\mathcal{B}$ avoids the generalized pattern $\bar{2}^{o}\underbracket[.5pt][1pt]{31}$, i.e.\ every descent in $\sigma$ (an occurrence of the generalized pattern $\underbracket[.5pt][1pt]{21}$) is part of an odd number of occurrences of $2\underbracket[.5pt][1pt]{31}$; equivalently, for any index $i$ such that $\sigma_i>\sigma_{i+1}$ there is an odd number of indices $j<i$ such that $\sigma_i>\sigma_j>\sigma_{i+1}$.
\end{itemize}

\begin{rem}
	Despite the very original nature of this second excluded pattern, the family $\mathcal{B}$ has some interesting aspects. For instance, it is in bijection with the family of lattice paths from $(0, 0)$ to $( n, \lfloor n/2\rfloor )$ with increments $(1, 0)$ and $(0, 1)$ that never go above the	line $y = x/2$. For a proof of this fact, see the discussion after \cite[Proposition 2.2]{MR2376115}. A second interesting aspect is highlighted in \cref{rem:sublattice}.
\end{rem}

\subsubsection*{The family $\mathcal A$}

From \cite[Proposition 2.1]{MR2376115} we have that the family $\mathcal A$ is encoded by the generating tree determined by the following succession rule
\begin{equation}\label{eq:gen_rule_fam_A}
\begin{cases} 
\text{Root label}: (1) \\
(k)\to (1),(2),\dots,(k-1),(k+1). 
\end{cases}
\end{equation}

Assumptions \ref{ass1} (with $\beta=1$) follows for the family $\mathcal A$ from the equation above (with the same slight modification as in the case of 123-avoiding permutations). 

Solving the system in \cref{eq:frhweibdowenpen} and building on the same ideas used for 123-avoiding permutations, we obtain the following.

\begin{prop}
	Assumption \ref{ass2} holds with $\alpha_y=2^{y}/3$, for all $y\in\Z_{<0}$, $\alpha_0=0$ and $\alpha_{1}=2/3$ and, $c_y=1$ for all $y\in\Z_{<0}\cup\{1\}$ and $c_0=0$.
\end{prop}

Assumption \ref{ass3} follows again from the specific behavior of the set $\text{AS}(\sigma)$ for $\sigma\in\mathcal{A}$ and how the statistic $|\text{AS}(\cdot)|$ determines the succession rule of the generating tree (for that we refer to the proof of \cite[Proposition 2.1]{MR2376115}). 

Therefore the main results of our paper also hold for the family $\mathcal{A}$.

\subsubsection*{The family $\mathcal B$}

From \cite[Proposition 2.2]{MR2376115} we have that the family $\mathcal B$ is encoded by the generating tree determined by the following succession rule
\begin{equation}\label{eq:gen_rule_fam_B}
\begin{cases} 
\text{Root label}: (1) \\
(k)\to (1+\mathds{1}_{\{k\text{ is odd}\}})\dots (k-3),(k-1),(k+1),
\end{cases}
\end{equation}
that is, the labels of the children of a vertex labeled by $k$ are the numbers $\ell$ such that $1 \leq \ell \leq k + 1$ and $k-\ell$ is odd.

Assumptions \ref{ass1} (with $\beta=1$) follows for the family $\mathcal B$ from \cref{eq:gen_rule_fam_B} (with the same slight modification as in the case of 123-avoiding permutations). 
Once again, solving the system in \cref{eq:frhweibdowenpen} and building on the same ideas used for 123-avoiding permutations, we obtain the following.

\begin{prop}
	 Assumption \ref{ass2} holds with $\alpha_y=2 \cdot3^{(-3 + y)/2 }\mathds{1}_{\{y\text{ is odd}\}}$, for all $y\in\Z_{\leq 1}$ and, $c_y=\mathds{1}_{\{y\text{ is odd}\}}$, for all $y\in\Z_{\leq0}$. 
\end{prop}

\begin{rem}\label{rem:sublattice}
	Note that the step-distribution $\alpha_y=2 \cdot3^{(-3 + y)/2 }\mathds{1}_{\{y\text{ is odd}\}}$ has span equal to $2$. 
\end{rem}

Assumption \ref{ass3} follows again using the specific behavior of the set $\text{AS}(\sigma)$ for $\sigma\in\mathcal{B}$ and how the statistic $|\text{AS}(\cdot)|$ determines the succession rule of the generating tree (for that we refer to the proof of \cite[Proposition 2.2]{MR2376115}). 

Therefore the main results of our paper hold for the family $\mathcal{B}$.

\medskip

Note that all the results in this \cref{sect:examples_gen_tree_Ok} give a proof of \cref{thm:examples_ok}.

\appendix

\section{Local topology for permutations}\label{sect:appe_local_topo}

This section summarizes the main results obtained in \cite[Section 2]{borga2018local}. In the first part of the section we define the notion of finite and infinite rooted permutation. Then we introduce a local distance and the corresponding notion of convergence for deterministic sequences of rooted permutations. Finally, we extend this notion of convergence (in two non-equivalent ways) to sequences of \emph{random} permutations.

We start by defining the notion of finite and infinite rooted permutation. 

\begin{defn}
	A \emph{finite rooted permutation} is a pair $(\sigma,i),$ where $\sigma\in\mathcal{S}^n$ and $i\in[n]$ for some $n\in\Z_{>0}.$
\end{defn}

We denote by $\mathcal{S}^n_{\bullet}$ the set of rooted permutations of size $n$ and with $\mathcal{S}_{\bullet}:=\bigcup_{n\in\Z_{>0}}\mathcal{S}^n_{\bullet}$ the set of finite rooted permutations. We write sequences of finite rooted permutations in $\mathcal{S}_{\bullet}$ as $(\sigma_n,i_n)_{n\in\Z_{>0}}.$ 

To a rooted permutation $(\sigma,i),$ we associate (as shown in the right-hand side of \cref{rest}) the pair $(\Asi,\leqsi),$  where $\Asi\coloneqq[-i+1,|\sigma|-i]$ is a finite interval containing 0 and $\leqsi$ is a total order on $\Asi,$ defined for all $\ell,j\in \Asi$ by
\begin{equation*}
\ell\leqsi j\qquad\text{if and only if}\qquad \sigma_{\ell+i}\leq\sigma_{j+i}\;.
\end{equation*} 
Informally, the elements of $\Asi$ should be thought of as the column indices of the diagram of $\sigma$,
shifted so that the root is in column $0$.
The order $\leqsi$ then corresponds to the vertical order on the dots in the corresponding columns.

\begin{figure}[htbp]
	\begin{center}
		\includegraphics[scale=.67]{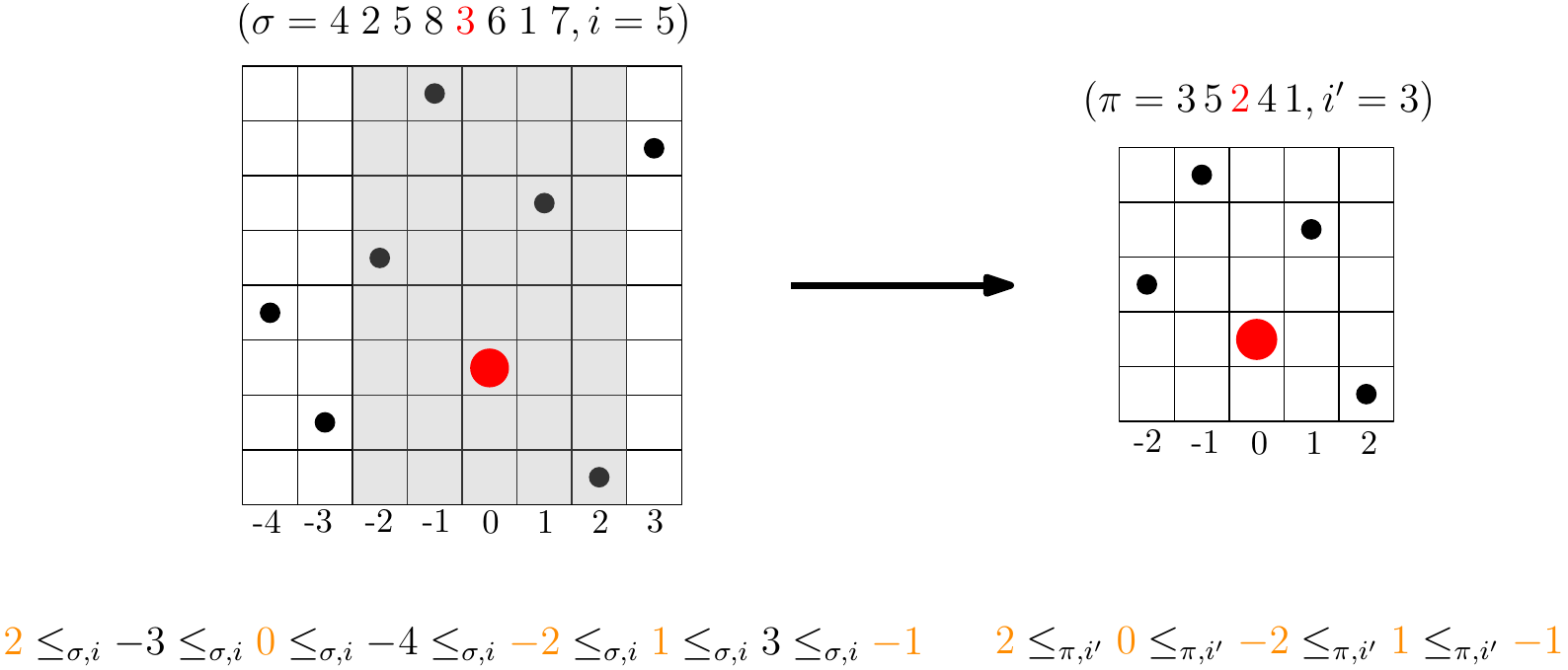}\\
		\caption{Two rooted permutations and the associated total orders. The big red dot indicates the root of the permutation. The vertical gray strip and the relation between the two rooted permutations will be clarified later.  \label{rest}}
	\end{center}
\end{figure}

Clearly this map is a bijection from the space of finite rooted permutations $\mathcal{S}_{\bullet}$ to the space of total orders on finite integer intervals containing zero. Consequently, we identify every rooted permutation  $(\sigma,i)$ with the total order $(\Asi,\leqsi).$

Thanks to the identification between rooted permutations and total orders, the following definition of infinite rooted permutation is natural. 

\begin{defn}\label{defn:inf_rooted}
	We call \emph{infinite rooted permutation} a pair $(A,\preccurlyeq)$ where $A$ is an infinite interval of integers containing 0 and $\preccurlyeq$ is a total order on $A$. We denote the set of infinite rooted permutations by $\mathcal{S}^{\infty}_\bullet.$
\end{defn} 
We highlight that infinite rooted permutations can be thought of as rooted at 0. We set 
$${\tilde{\mathcal{S}}_{\bullet}}\coloneqq\Sr\cup\mathcal{S}^\infty_{\bullet},$$
namely, the set of finite and infinite rooted permutations. 

We now introduce the following \textit{restriction function around the root} defined, for every $h\in\Z_{>0}$, as follows
\begin{equation}
\label{rhfunct}
\begin{split}
r_h \colon\quad &{\tilde{\mathcal{S}}_{\bullet}}\;\longrightarrow \qquad \;\mathcal{S}_\bullet;\\
(A,&\preccurlyeq) \mapsto \big(A\cap[-h,h],\preccurlyeq\big)\;.
\end{split}
\end{equation}
We can think of restriction functions as a notion of neighborhood around the root.
For finite rooted permutations we also have the equivalent description of the restriction functions $r_h$ in terms of consecutive patterns: if $(\sigma,i)\in\Sr$ then $r_h(\sigma,i)=\pat_{[a,b]}(\sigma,i)$ where $a=\max\{1,i-h\}$ and $b=\min\{|\sigma|,i+h\}.$ An example is given in \cref{rest} for the restriction function $r_2$.

The \emph{local distance} $d$ on the set of (possibly infinite) rooted permutations ${\tilde{\mathcal{S}}_{\bullet}}$ is defined as follows: given two rooted permutations $(A_1,\preccurlyeq_1),(A_2,\preccurlyeq_2)\in{\tilde{\mathcal{S}}_{\bullet}},$
\begin{equation}
\label{distance}
d\big((A_1,\preccurlyeq_1),(A_2,\preccurlyeq_2)\big)\coloneqq2^{-\sup\big\{h\in\Z_{>0}\;:\;r_h(A_1,\preccurlyeq_1)=r_h(A_2,\preccurlyeq_2)\big\}},
\end{equation}
with the conventions that $\sup\emptyset=0,$ $\sup\Z_{>0}=+\infty$ and $2^{-\infty}=0.$ Obviously, the distance $d$ determines a notion of convergence for deterministic sequences of rooted permutations, called \emph{local convergence}.
The metric space $({\tilde{\mathcal{S}}_{\bullet}},d)$ is a compact space (see \cite[Theorem 2.16]{borga2018local}).

We also recall the definitions of two different versions of local convergence for sequences of random permutations rooted uniformly at random. The presence of two sources of randomness, one for the choice of the permutation and one for the (uniform) choice of the root, leads to two non-equivalent possible definitions: the \emph{annealed} and the \emph{quenched}  version of the so-called Benjamini--Schramm convergence. Intuitively, in the second definition, the random permutation is frozen, whereas  in the first one, the random permutation and the choice of the root are treated on the same level.  

In both cases, $(\bm{\sigma}_n)_{n\in\Z_{>0}}$ denotes a sequence of random permutations in $\mathcal{S}$ and $\bm{i}_n$ denotes a uniform index of $\bm{\sigma}_n,$ i.e. a uniform integer in $[1,|\bm{\sigma}_n|]$.
\begin{defn}[Annealed version of the Benjamini-Schramm convergence]\label{defn:weakweakconv}
	We say that $(\bm{\sigma}_n)_{n\in\Z_{>0}}$ \emph{converges in the annealed Benjamini-Schramm sense} to a random variable $\bm{\sigma}^{\infty}$ with values in $\Sri$ if the sequence of random variables $(\bm{\sigma}_n,\bm{i}_n)_{n\in\Z_{>0}}$ converges in distribution to $\bm{\sigma}^{\infty}$ with respect to the local distance $d$.
\end{defn}

\begin{defn}[Quenched version of the Benjamini-Schramm convergence]\label{strongconv}
	We say that $(\bm{\sigma}_n)_{n\in\Z_{>0}}$ \emph{converges in the quenched Benjamini-Schramm sense} 
	to a random measure $\bm{\mu}^\infty$ on $\Sri$
	if the sequence of conditional laws $\big(\mathcal{L}aw\big((\bm{\sigma}_n,\bm{i}_n)\big|\bm{\sigma}_n\big)\big)_{n\in\Z_{>0}}$
	converges in distribution to $\bm{\mu}^\infty$ with respect to the weak topology induced by the local distance $d$.
\end{defn}

\begin{rem}A more combinatorial description of the above definition can be found in \cite[Section 2.5]{borga2018local}.
\end{rem}
We highlight that, in the annealed version, the limiting object is a {\em random variable} with values in $\Sri$,
while for the quenched version, the limiting object $\bm{\mu}^{\infty}$ is a {\em random measure} on $\Sri$.

We have the following characterizations of the two versions of the Benjamini-Schramm convergence 
\cite[Theorems 2.24 and 2.32]{borga2018local}.

\begin{thm}
	\label{thm:local_conv_perm_charact}
	For any $n\in\Z_{>0}$, let $\bm{\sigma}_n$ be a random permutation of size $n$. Then
	\begin{enumerate}
		\item The sequence $(\bm{\sigma}_n)_{n\in\Z_{>0}}$ converges in the annealed Benjamini-Schramm sense
		to some $\bm{\sigma}^{\infty}$ if and only if
		there exist non-negative real numbers $(\Delta_{\pi})_{\pi\in\mathcal{S}}$ such that 
		$$\E\left[\frac{{\coc}(\pi,\bm{\sigma}_n)}{n}\right]\to\Delta_{\pi},\quad\text{for all patterns}\quad\pi\in\mathcal{S}.$$
		\item The sequence $(\bm{\sigma}_n)_{n\in\Z_{>0}}$ converges in the quenched Benjamini-Schramm sense
		to some $\bm{\mu}^\infty$ if and only if      
		there exist non-negative real random variables $(\bm{\Lambda}_\pi)_{\pi\in\mathcal{S}}$ such that
		$$\left(\frac{\coc(\pi,\bm{\sigma}_n)}{n}\right)_{\pi\in\mathcal{S}}\stackrel{d}{\to}(\bm{\Lambda}_{\pi})_{\pi\in\mathcal{S}},$$
		w.r.t. the product topology. 
	\end{enumerate}
Since the variables $\frac{\coc(\pi,\bm{\sigma}_n)}{n}$ take values in $[0,1]$,
the quenched Benjamini-Schramm convergence implies the annealed one. Moreover, if $(2)$ holds (and hence also $(1)$) then, $\E[\bm{\Lambda}_{\pi}]=\Delta_{\pi},$ for all $\pi\in\mathcal{S}.$
\end{thm}

We conclude this section with the particular case when the limiting objects $(\bm{\Lambda}_{\pi})_{\pi\in\mathcal{S}}$ in \cref{thm:local_conv_perm_charact} are deterministic (see \cite[Corollary 2.38]{borga2018local}). 

\begin{cor}
	\label{detstrongbsconditions}
	For any $n\in\Z_{>0},$ let $\bm{\sigma}^n$ be a random permutation of size $n$. Then the following are equivalent:
	\begin{enumerate}[(a)]
		\item there exists a (deterministic) measure $\mu$ on $\Sri$ that is the quenched Benjamini-Schramm limit of $\bm{\sigma}^n;$
		\item there exists an infinite vector of non-negative real numbers $(\Lambda_{\pi})_{\pi\in\mathcal{S}}$ such that for all $\pi\in\mathcal{S},$ $$\frac{\coc(\pi,\bm{\sigma}^n)}{n}\stackrel{P}{\to}\Lambda_{\pi}.$$
	\end{enumerate}
\end{cor}

\section*{Acknowledgements}
The author is very grateful to Mathilde Bouvel and Valentin Féray for constant and stimulating discussions. Thanks to Svante Janson, Vitali Wachtel, Francesco Caravenna and Igor Kortchemski for some discussions and clarifications on conditioned random walks. He finally thanks the anonymous referees for all their precious and helpful comments.

This work was completed with the support of the SNF Eccellenza grant no.\ PCEFP2\_18687, entitled  ``Enumeration and limit shapes for pattern-avoiding permutations and related combinatorial objects".

\bibliography{mybib}
\bibliographystyle{alpha}

\end{document}